\documentclass[12pt,reqno]{amsart}
\usepackage[margin=1in]{geometry}
\usepackage{amsmath,amssymb,amsthm,graphicx,amsxtra, setspace}
\usepackage[utf8]{inputenc}
\usepackage{mathrsfs}
\usepackage{hyperref}
\usepackage{xcolor}
\usepackage{upgreek}
\usepackage{mathtools}
\allowdisplaybreaks
\newtheorem{theorem}{Theorem}[section]
\newtheorem{lem}[theorem]{Lemma}
\newtheorem{rem}[theorem]{Remark}
\newtheorem{cor}[theorem]{Corollary}
\newtheorem{Def}[theorem]{Definition}

\newtheorem{Ex}[theorem]{Example}
\newtheorem{Ass}[theorem]{Assumption}
\usepackage{latexsym}
\usepackage{hyperref}
\usepackage{graphicx}
\usepackage{epstopdf}

\DeclareMathOperator*{\esssup}{ess\,sup}

\allowdisplaybreaks

\let\originalleft\left
\let\originalright\right
\renewcommand{\left}{\mathopen{}\mathclose\bgroup\originalleft}
\renewcommand{\right}{\aftergroup\egroup\originalright}

\newcommand{\Addresses}{{
		\footnote{
			\footnotesize
			\noindent \textsuperscript{*}Corresponding author.
			
			\noindent  S. Arora\\ \textit{e-mail:} \texttt{arorasumit10623@gmail.com}
			
			\noindent  A.K. Nandakumaran\\ \textit{e-mail:} \texttt{nands@iisc.ac.in}
			
			\noindent	Department of Mathematics, Indian Institute of Science, Bangalore 560012, India.

			\textit{Key words:} approximate controllability, neutral integro-differential equation, resolvent family, fixed point theorem.
			
			Mathematics Subject Classification (2020): 34K06, 34A12, 37L05, 93B05.

}}}

\begin{document}
	\title[Neutral integro-differential equation]{Controllability problems of a neutral integro-differential equation with memory \Addresses}
	\author [S. Arora and A.K. Nandakumaran]{Sumit Arora and Akambadath Nandakumaran\textsuperscript{*}}
	\maketitle{}
	\begin{abstract} The current study addresses the control problems posed by a semilinear neutral integro-differential equation with memory. The primary objectives of this study are to investigate the existence of a mild solution and approximate controllability of both linear and semilinear control systems in Banach spaces. To accomplish this, we begin by introducing the concept of a resolvent family associated with the homogeneous neutral integro-differential equation without memory. In the process, we establish some important properties of the resolvent family. Subsequently, we develop approximate controllability results for a linear control problem by constructing a linear-quadratic regulator problem. This includes establishing the existence of an optimal pair and determining the expression of the optimal control that produces the approximate controllability of the linear system. Furthermore, we deduce sufficient conditions for the existence of a mild solution and the approximate controllability of a semilinear system in a reflexive Banach space with a uniformly convex dual. Additionally, we delve into the discussion of the approximate controllability for a semilinear problem in general Banach space, assuming a Lipschitz type condition on the nonlinear term. Finally, we implement our findings to examine the approximate controllability of certain partial differential equations, thereby demonstrating their practical relevance.
	\end{abstract}
	\section{Introduction}\label{intro}\setcounter{equation}{0}
	Let $\mathbb{W}$ be a Banach space with its dual $\mathbb{W^{*}}$, and  $\mathbb{U}$ be a Hilbert space (identified with its own dual). We focus on analyzing the existence and approximate controllability of the following abstract neutral integro-differential equations: 
	\begin{equation}\label{SEq}
		\left\{
		\begin{aligned}
			\frac{\mathrm{d}}{\mathrm{d}t}\left[w(t)+\int_{-\infty}^{t}\mathrm{G}(t-s)w(s)\mathrm{d}s\right]&=\mathrm{A}w(t)+\int_{-\infty}^{t}\mathrm{N}(t-s)w(s)\mathrm{d}s+\mathrm{B}u(t)\\&\qquad+f(t,w_t), \ t\in(0,T],\\
			w_{0}&=\psi\in \mathfrak{B}.
		\end{aligned}
		\right.
	\end{equation}
Here $\mathrm{A}:D(\mathrm{A})\subseteq\mathbb{W}\to\mathbb{W}$ and  $\mathrm{N}(t):D(\mathrm{N}(t))\subseteq\mathbb{W}\to\mathbb{W}$ for $t\ge0$, are closed linear operators. The linear operators $\mathrm{G}(t):\mathbb{W}\to\mathbb{W}, t\ge 0$ and $\mathrm{B}:\mathbb{U}\to\mathbb{W}$ are bounded. The conditions on the  nonlinear function $ f:J\times \mathfrak{B}\rightarrow \mathbb{W} $, where $\mathfrak{B}$ denotes the phase space, will be explained in the subsequent section. The function $w_{t}:(-\infty, 0]\rightarrow\mathbb{W}$ with $w_{t}(\theta)=w(t+\theta)$ and $w_t\in\mathfrak{B}$ for each $t\ge 0$ and the control function $u$ is in $\mathrm{L}^2(J;\mathbb{U})$.
	
	
	The classical heat equation (where both internal energy and heat flux exhibit a linear dependence on the temperature $w$ and its gradient $\nabla w$) adequately describes the evolution of temperature across various materials, such as copper, silicon carbide, polystyrene, aluminum, and many others. However, this description does not provide a complete understanding of how heat diffuses in materials with fading memory, mainly because in the classical model it is assumed that changes in the heat source immediately affects the material. The works proposed by Gurtin and Pipkin \cite{ARMA1968} and Nunziato \cite{QAM1971} related to heat conduction phenomena in materials with fading memory, considered the internal energy and heat flux as functionals of $w$ and $\nabla w$. This theory is adequate for describing heat conduction in materials exhibiting fading memory. Subsequently, in the following studies \cite{NDEA2003,SJMA1981,SIAM1990}, neutral integro-differential systems have been frequently employed to describe heat flow phenomena in various materials with fading memory of the form
	\begin{equation*}
		\left\{
		\begin{aligned}
			&\frac{\mathrm{d}}{\mathrm{d}t}\left[w(t,\xi)+\int_{-\infty}^{t}k_1(t-s)w(s,\xi)\mathrm{d}s\right]=c\Delta w(t,\xi)\nonumber\\&\quad+\int_{-\infty}^{t}k_2(t-s)\Delta w(s,\xi)\mathrm{d}s+h(t,\xi, w(t,\xi)),t\ge0, \xi\in\Omega,\\
			&w(0,\xi)=0,\ \xi\in\partial\Omega, 
		\end{aligned}
		\right.
	\end{equation*}

	where $w(t,\xi)$  represents the value at position $\xi$ and time $t$ for $(t,\xi)\in[0,\infty)\times\Omega$ with open and bounded domain $\Omega\in\mathbb{R}^n$ has a boundary $\partial\Omega$ of class $C^2$. The constant $c$ denotes the physical constant and $k_i:\mathbb{R}\to\mathbb{R}, i=1,2$, are the internal energy and the heal flux relaxation, respectively. For further exploration into various models of partial integro-differential equations and their associated applications, we refer to \cite{S1996}. A significant approach to deal with such kinds of systems is to transform them into integro-differential evolution equations in abstract spaces. In fact, if we assume the value of solution $w$ is known on the interval $(-\infty,0]$, we can then transform the above system into the form
	\begin{align*}
		\frac{\mathrm{d}}{\mathrm{d}t}\left[w(t)+F(t,w_t)\right]=\mathrm{A}w(t)+\int_{0}^{t}K(t-s)w(s)\mathrm{d}s+f(t,w_t), t\ge 0.
	\end{align*}
Moreover, the model proposed by Coleman and Gurtin \cite{ZAMP1967} and Miller \cite{JMAA1978} for heat conduction in a rigid, isotropic viscoelastic material under elastic conditions was studied in \cite{JDE1988} using abstract evolution equations. It is observed that the theory of the resolvent operator plays a crucial role in studying semilinear integro-differential evolution equations, see \cite{SIJMA1984, TAMS1982,BVB1993} etc.
	
	On the other hand, controllability, whether exact or approximate, is widely acknowledged in both engineering and mathematical control theory. Controllability captures the ability of a solution to a control problem that start from any initial state to a desired target state through suitable controls. The exact controllability problem has been studied in various instances, see for example, \cite{JDEKDV2020, EECT2015, AMO2022} and the references therein. The approximate controllability problem, particularly in infinite-dimensional systems, has received significant attention due to its broad applications (cf. \cite{SIAM2003,TRR,TR}). Many researchers have achieved notable results on approximate controllability for nonlinear systems in Hilbert and Banach spaces using the resolvent operator condition (e.g., \cite{MCRF2021, EECT2017, ZL2015}).
	
	Recent research on semilinear integro-differential evolution equations, using resolvent operator theory, has yielded significant advancements in existence, stability, regularity, and control problems etc,. see for instance, \cite{NA2010,AAM2017,IJC2018,JFPTA2020} and the references provided therein. In particular, Dos Santos et al. \cite{JIEA2011} established the theory of resolvent operators for the following linear neutral integro-differential equation of the form
	\begin{equation*}
		\left\{
		\begin{aligned}
			\frac{\mathrm{d}}{\mathrm{d}t}\left[w(t)+\int_{0}^{t}\mathrm{G}(t-s)w(s)\mathrm{d}s\right]&=\mathrm{A}w(t)+\int_{0}^{t}\mathrm{N}(t-s)w(s)\mathrm{d}s,\ t\in(0,T],\\
			w(0)&=\zeta\in\mathbb{W}.
		\end{aligned}
		\right.
	\end{equation*}
	They also discussed the existence of mild, strict and classical solutions for the problem of the form given in \eqref{SEq} without considering control. Furthermore, this theory has been increasingly applied in recent years to investigate various neutral partial integro-differential equations, as demonstrated in \cite{JFPTA2023,JIEA2013,EECT2022,JMAA2018} and many other works.
	
	Recently the approximate controllability problem for neutral integro-differential equations in Hilbert spaces include  Mokkedem and Fu \cite{AMC2014}, who used fractional power operator theory, fixed point techniques, and the resolvent operator condition. Nonlocal neutral integro - differential systems with impulses and finite delays were explored in \cite{MMAS2021} using a resolvent operator theory and approximation methods. Furthermore, Cao and Fu \cite{JIEA2022} determined sufficient conditions for the approximate controllability of neutral semilinear integro-differential equations driven by fractional Brownian motion.
	
	To the best of our knowledge, no results have been demonstrated on the approximate controllability of neutral semilinear integro-differential equations in Banach spaces, particularly through the utilization of the resolvent operator condition. This work fills this gap by studying such systems in a reflexive Banach space with uniformly convex dual. Initially, we establish important properties of the resolvent family that are essential for subsequent developments. We then analyze the approximate controllability of a linear problem (see Section \ref{LCS}). Subsequently, we derive sufficient conditions for the approximate controllability of our semilinear system \eqref{SEq}. Furthermore, we extend the results within the framework of general Banach spaces under a Lipschitz type condition on the nonlinear term $f(\cdot,\cdot)$. Unlike existing studies \cite{JIEA2022,MMAS2021,AMC2014}, this work avoids the fractional power theory of linear operators, which is often used in the literature on integro-differential equations.
	\section{Preliminaries}\label{pre}\setcounter{equation}{0}
	Assume that the duality pairing between $\mathbb{W}$ and its dual $\mathbb{W^{*}}$ is represented by $\langle \cdot, \cdot  \rangle $. The notation $\mathcal{L}(\mathbb{U};\mathbb{W})$ and $\mathcal{L}(\mathbb{W})$, respectively stand for the space of all bounded linear operators from $\mathbb{U}$ to $\mathbb{W}$, endowed with the operator norm $\|\cdot\|_{\mathcal{L}(\mathbb{U};\mathbb{W})}$ and  the space of all bounded linear operators on $\mathbb{W}$, equipped with the norm $\|\cdot\|_{\mathcal{L}(\mathbb{W})}$. For a closed linear operator $\mathrm{P}:D(\mathrm{P})\subseteq\mathbb{W}\to\mathbb{Y}$, the notation $D_{1}(P)$ represents the domain of $P$ endowed with the graph norm $\|z\|_{1}=\|z\|_{\mathbb{W}}+\|\mathrm{P}z\|_{\mathbb{Y}}$. The Laplace transform of an appropriate function $ \mathrm{J}:[0,\infty) \to \mathbb{W}$ is denoted by $\widehat{J}$. We now introduce the resolvent  and some of its properties.
	\subsection{Resolvent operator} In this subsection, we introduce the concept of a resolvent operator for the following abstract Cauchy problem involving integro-differential equations:
	\begin{equation}\label{LEq}
		\left\{
		\begin{aligned}
			\frac{\mathrm{d}}{\mathrm{d}t}\left[w(t)+\int_{0}^{t}\mathrm{G}(t-s)w(s)\mathrm{d}s\right]&=\mathrm{A}w(t)+\int_{0}^{t}\mathrm{N}(t-s)w(s)\mathrm{d}s,\ t\in(0,T],\\
			w(0)&=\zeta\in\mathbb{W}.
		\end{aligned}
		\right.
	\end{equation}
	\begin{Def}
		A one-parameter family of bounded linear operators $(\mathscr{R}(t))_{t \ge 0}$ on $\mathbb{W}$ is called a resolvent operator of \eqref{LEq} if the following condition are satisfied:
		\begin{itemize}
			\item [(a)] The function $\mathscr{R}:[0,\infty)\to \mathcal{L}(\mathbb{W})$ is exponentially bounded, strongly continuous and $\mathscr{R}(0)z=z$ for all $z\in\mathbb{W}$. 
			\item [(b)] For $z\in D(\mathrm{A}), \mathscr{R}(\cdot)z\in C([0,\infty);D_1(\mathrm{A}))\cap C^1((0,\infty);\mathbb{W})$, and 
			\begin{align}
				\frac{\mathrm{d}}{\mathrm{d}t}\left[\mathscr{R}(t)z+\int_{0}^{t}\mathrm{G}(t-s)\mathscr{R}(s)z\mathrm{d}s\right]=\mathrm{A}\mathscr{R}(t)z+\int_{0}^{t}\mathrm{N}(t-s)\mathscr{R}(s)z\mathrm{d}s\label{2.3},\\
				\frac{\mathrm{d}}{\mathrm{d}t}\left[\mathscr{R}(t)z+\int_{0}^{t}\mathscr{R}(t-s)\mathrm{G}(s)z\mathrm{d}s\right]=\mathscr{R}(t)\mathrm{A}z+\int_{0}^{t}\mathscr{R}(t-s)\mathrm{N}(s)z\mathrm{d}s\label{2.4},		
			\end{align}
			for every $t\ge 0$.
		\end{itemize}
	\end{Def}
	Throughout this study, we assume that the following conditions are verified.
	\begin{enumerate}
		\item [\textbf{\textit{(Cd1)}}] The linear operator $\mathrm{A}:D(\mathrm{A})\subseteq\mathbb{W}\to\mathbb{W}$ is the infinitesimal generator of an analytic semigroup $(\mathcal{S}(t))_{t\ge 0}$. Let us take two constants $M>0$ and $\nu\in(\pi/2,\pi)$ such that $\rho(A)\supseteq\Lambda_{\nu}=\{\lambda\in\mathrm{\textbf{C}}\setminus\{0\}:|\arg(\lambda)|<\nu\}$ and $\|\mathrm{R}(\lambda,\mathrm{A})\|_{\mathcal{L}(\mathbb{W})}\le M/|\lambda|$ for all $\lambda\in\Lambda_{\nu}$.
		\item [\textbf{\textit{(Cd2)}}] The function $\mathrm{G}:[0,\infty)\to \mathcal{L}(\mathbb{W})$ is strongly continuous and $\widehat{\mathrm{G}}(\lambda)z$ is absolutely convergent for $z\in\mathbb{W}$ and $\mathrm{Re}(\lambda)>0$. Additionally, there exists a constant $\alpha>0$ and an analytical extension of $\widehat{\mathrm{G}}(\lambda)$ (still denoted by $\widehat{\mathrm{G}}(\lambda)$) to $\Lambda_{\nu}$ such that $\|\widehat{\mathrm{G}}(\lambda)\|_{\mathcal{L}(\mathbb{W})}\le N_1|\lambda|^{-\alpha}$ for every $\lambda\in\Lambda_{\nu}$, and $\|\widehat{\mathrm{G}}(\lambda)z\|_{\mathbb{W}}\le N_2|\lambda|^{-1}\|z\|_{1} $ for every $\lambda\in\Lambda_{\nu}$ and $z\in D(\mathrm{A})$.
		\item [\textbf{\textit{(Cd3)}}] The mapping $\mathrm{N}(t):D(\mathrm{N}(t))\subseteq\mathbb{W}\to\mathbb{W}$ is a closed linear operator for each $t\ge0$, the domain $D(\mathrm{A})\subseteq D(\mathrm{N}(t))$, and $\mathrm{N}(\cdot)z$ is strongly measurable on $(0,\infty)$ for each $z\in D(\mathrm{A})$. 
		Moreover, there is a function $\varPi\in\mathrm{L}^1_{loc}(\mathbb{R}^+)$ such that $\widehat{\varPi}(\lambda)$ exists for $\mathrm{Re}(\lambda)>0$ and $\|\mathrm{N}(t)z\|_{\mathbb{W}}\le\varPi(t)\|z\|_{1}$ for all $t>0$ and $z\in D(\mathrm{A})$. In addition, the operator valued function $\widehat{\mathrm{N}}:\Lambda_{\pi/2}\to\mathcal{L}(D_1(\mathrm{A}),\mathbb{W})$ has an analytic extension (still denoted by $\widehat{\mathrm{N}}$) to $\Lambda_{\nu}$ such that $\|\widehat{\mathrm{N}}(\lambda)z\|_{\mathbb{W}}\le\|\widehat{\mathrm{N}}(\lambda)\|_{\mathcal{L}(D_1(\mathrm{A});\mathbb{W})}\|z\|_{1}$ for each $z\in D(\mathrm{A})$, and $\|\widehat{\mathrm{N}}(\lambda)\|_{\mathcal{L}(D_1(\mathrm{A});\mathbb{W})}\to 0$ as $| \lambda| \to\infty$.
		\item [\textbf{\textit{(Cd4)}}] There exists a subspace $Y\subseteq D(\mathrm{A})$, which is dense in the graph norm, and two positive constants $M_i, i=1,2,$ such that $\mathrm{A}(Y)\subseteq D(\mathrm{A}),\ \widehat{\mathrm{N}}(\lambda)(Y)\subseteq D(\mathrm{A}),\ \widehat{\mathrm{G}}(Y)\subseteq D(\mathrm{A}),\ \|\mathrm{A}\widehat{\mathrm{N}}(\lambda)z\|_{\mathbb{W}}\le M_1\|z\|_{1}$ and $\|\widehat{\mathrm{G}}(\lambda)z\|_{1}\le M_2|\lambda|^{-\alpha}\|z\|_{1}$ for all $\lambda\in\Lambda_{\nu}$ and every $z\in Y$.
	\end{enumerate}
	In the sequel, let us define a set  $\Lambda_{r,\theta}=\{\lambda\in\mathrm{\textbf{C}}\setminus\{0\}:|\lambda|>r, |\arg(\lambda)|<\theta\}$ for $r>0$ and $\theta\in(\pi/2,\nu)$. We also define $\Gamma_{r,\theta}=\bigcup_{k=1}^{3} \Gamma^{k}_{r,\theta}$, where $$\Gamma^{1}_{r,\theta}=\{te^{i\theta}:t\ge r\},\ \Gamma^{2}_{r,\theta}=\{re^{i\phi}: -\theta\le\phi\le\theta\},\ \Gamma^{3}_{r,\theta}=\{te^{-i\theta}:t\ge r\}.$$ The orientation is considered as counterclockwise. In addition, let $$\Omega(\mathrm{F})=\{\lambda\in\mathrm{\textbf{C}}:\mathrm{F}(\lambda):=(\lambda I+\lambda\hat{\mathrm{G}}(\lambda)-\mathrm{A}-\hat{\mathrm{N}}(\lambda))^{-1}\in\mathcal{L}(\mathbb{W})\}.$$
	We now review some important properties of the resolvent operator $(\mathscr{R}(t))_{t \ge 0}$ that are needed to establish our results.
	\begin{lem}[Lemma 2.2 \cite{JIEA2011}]\label{lem2.2} There exists a constant $\tilde{r}>0$ such that $\Lambda_{\tilde{r},\theta}\subseteq\Omega(\mathrm{F})$, and the mapping $\mathrm{F}(\lambda):\Lambda_{\tilde{r},\theta}\to\mathcal{L}(\mathbb{W})$ is analytic. Moreover, there exists a constant $M_3>0$ such that $$\|\lambda\mathrm{F}(\lambda)\|_{\mathcal{L}(\mathbb{W})}\le M_3, \ \mbox{for all}\ \lambda\in\Lambda_{\tilde{r},\theta}.$$
	\end{lem}
	\begin{rem}
		Note that if $(\mathscr{R}(t))_{t\ge 0}$ is a resolvent operator of \eqref{LEq}, it follows from \eqref{2.4} that $\widehat{\mathscr{R}}(\lambda)F(\lambda)z=z$ for all $z\in D(\mathrm{A})$. By employing Lemma \ref{lem2.2} and the properties of Laplace transformation, we deduce that  resolvent operator for \eqref{LEq} is unique if exists.
	\end{rem}
	\begin{theorem}[Theorem 2.1 \cite{JIEA2011}]\label{thm2.3} Let conditions \textbf{(Cd1)}-\textbf{(Cd4)} be true. Then there exists a unique resolvent operator $\mathscr{R}(\cdot)$ defined as 
		\begin{equation}\label{RSL}
			\mathscr{R}(t)=
			\begin{dcases}
				\frac{1}{2\pi i}\int_{\Gamma_{r,\theta}}e^{\lambda t}\mathrm{F}(\lambda)\mathrm{d}\lambda, \ t>0,\\
				\mathrm{I}, \ \qquad\qquad\qquad\qquad\quad t=0,
			\end{dcases}
		\end{equation}
		of the system \eqref{LEq}, where $r>\tilde{r}$ and $\theta\in(\pi/2,\nu)$.
	\end{theorem}
	\begin{lem}[Lemma 2.3 and Lemma 2.5 \cite{JIEA2011}]\label{lem2.5} The mapping $t\to\mathscr{R}(t),\ t\ge 0$ is strongly continuous in $\mathbb{W}$ and exponentially bounded in $\mathcal{L}(\mathbb{W})$, i.e. $\|\mathscr{R}(t)\|_{\mathcal{L}(\mathbb{W})}\le Ce^{\omega t}$ for some $C>0$ and $\omega>0$.
	\end{lem}
	\begin{lem}[Lemma 2.4 and Lemma 2.6 \cite{JIEA2011}]\label{lemm2.5} The mapping $t\to\mathscr{R}(t),\ t\ge 0$ is strongly continuous in $D_1(\mathrm{A})$ and exponentially bounded in $\mathcal{L}(D_1(\mathrm{A}))$.
	\end{lem}
	\begin{lem}[Lemma 3.11 \cite{JIEA2011}]\label{lem2.6}
		Let $\mathrm{R}(\lambda_0,\mathrm{A})$ is compact for some $\lambda_0\in\rho(\mathrm{A})$. Then the operator $\mathscr{R}(t)$ is compact for all $t>0$.
	\end{lem}
	\begin{lem}\label{lem2.7}
		The operator $\mathscr{R}(t)$ is continuous in $\mathcal{L}(\mathbb{W})$ for $t>0$.
	\end{lem}
	\begin{proof}
		For $0<s<t$, let us compute 
		\begin{align*}
			\left\|\mathscr{R}(t)-\mathscr{R}(s)\right\|_{\mathcal{L}(\mathbb{W})}&=\frac{1}{2\pi}\left\|\int_{\Gamma_{r,\theta}}\left(e^{\lambda t}-e^{\lambda s}\right)\mathrm{F}(\lambda)\mathrm{d}\lambda\right\|_{\mathcal{L}(\mathbb{W})}\nonumber\\&\le \frac{M_3}{2\pi}\int_{\Gamma_{r,\theta}}\left|\frac{e^{\lambda t}-e^{\lambda s}}{\lambda} \right||\mathrm{d}\lambda|.
		\end{align*}
	Now we claim that, there exist $s<c<t$ such that 
			\begin{align}\label{e2.5}
				\left|\frac{e^{\lambda t}-e^{\lambda s}}{\lambda} \right|\le (t-s)|e^{\lambda c}|,\ \mbox{for}\ \lambda\in\Gamma_{r,\theta}.\end{align}
			To prove the claim, let us define a function
			$$ g(r)=(\mathrm{Re}(e^{\lambda t})-\mathrm{Re}(e^{\lambda s}))\mathrm{Re}(e^{\lambda r})+(\mathrm{Im}(e^{\lambda t})-\mathrm{Im}(e^{\lambda s}))\mathrm{Im}(e^{\lambda r}),$$ for  $s\le r\le t$ and $\lambda\in\Gamma_{r,\theta}$.
		The function $g(\cdot)$ is real-valued,  continuous on $[s,t]$ and differentiable on $(s,t)$. Then by the mean value theorem
			\begin{align}\label{e2.6}g(t)-g(s)=g'(c)(t-s), \ \mbox{for some} \ c\in(s,t).\end{align}
			Further, we compute 
			\begin{align}\label{e2.77}
				g(t)-g(s)&= \left(\mathrm{Re}(e^{\lambda t})-\mathrm{Re}(e^{\lambda s})\right)^{2}+\left(\mathrm{Im}(e^{\lambda t})-\mathrm{Im}(e^{\lambda s})\right)^{2}\nonumber\\&=|e^{\lambda t}-e^{\lambda s}|^{2}
			\end{align}
			and 
			\begin{align*}
				g'(c)&=\left(\mathrm{Re}(e^{\lambda t})-\mathrm{Re}(e^{\lambda s})\right)\left(\mathrm{Re}(\lambda)\mathrm{Re}(e^{\lambda c})-\mathrm{Im}(\lambda)\mathrm{Im}(e^{\lambda c})\right)\nonumber\\&\quad+(\mathrm{Im}(e^{\lambda t})-\mathrm{Im}(e^{\lambda s}))\left(\mathrm{Re}(\lambda)\mathrm{Im}(e^{\lambda c})+\mathrm{Im}(\lambda)\mathrm{Re}(e^{\lambda c})\right)
			\end{align*}
			From the above expression, we deduce  that 
			\begin{align}\label{e2.88}
				|g'(c)|\le |\lambda||e^{\lambda t}-e^{\lambda s}| |e^{\lambda c}|
			\end{align}
			Combining \eqref{e2.6}, \eqref{e2.77} and \eqref{e2.88}, we obtain  the claim  \eqref{e2.5}. Then \eqref{e2.5} implies that
			\begin{align}\label{2.5}\left|\frac{e^{\lambda t}-e^{\lambda s}}{\lambda} \right|\to 0, \ \mbox{as}\ t\to s.\end{align}
		We now evaluate
		\begin{align*}
			\int_{\Gamma_{r,\theta}}\frac{M_3}{|\lambda|}e^{\mathrm{Re}(\lambda)t}|\mathrm{d}\lambda|&=2M_3\int_{r}^{\infty}\frac{e^{ts\cos\theta}}{s}\mathrm{d}s+M_3\int_{-\theta}^{\theta}e^{tr\cos\xi}\mathrm{d}\xi\nonumber\\&\le \frac{2M_3}{r|\cos\theta|}+2M_3\theta e^{rt}<\infty, \ \mbox{for each}\ t>0.
		\end{align*}
		Using the aforementioned fact along with the convergence \eqref{2.5} and the Lebesgue dominated convergence theorem, we see that
	$\|\mathscr{R}(t)-\mathscr{R}(s)\|_{\mathcal{L}(\mathbb{W})}\to 0, \ \mbox{as}\ t\to s$
		proving the operator $\mathscr{R}(t)$ is continuous in $\mathcal{L}(\mathbb{W})$ for $t>0$.
	\end{proof}
	Throughout this work, we take $\sup_{t\in[0,T]}\|\mathscr{R}(t)\|_{\mathcal{L}(\mathbb{W})}\le K$.
	\begin{lem}\label{lem2.21}
		If the resolvent operator $\mathscr{R}(t)$ is compact for $t>0$, then the following assertions hold:
		\begin{itemize}
			\item [(a)]$\lim\limits_{h\to 0^+}\|\mathscr{R}(t+h)-\mathscr{R}(h)\mathscr{R}(t)\|_{\mathcal{L}(\mathbb{W})}=0$ for all $t>0$.
			\item [(b)]$\lim\limits_{h\to 0^+}\|\mathscr{R}(t)-\mathscr{R}(h)\mathscr{R}(t-h)\|_{\mathcal{L}(\mathbb{W})}=0$ for all $t>0$.
		\end{itemize} 
	\end{lem}
	\begin{proof}
		Let us first prove (a). We take  $z\in\mathbb{W}$ with $\|z\|_{\mathbb{W}}\le 1$ and any $\epsilon>0$. Then by making use of the compactness of the operator $\mathscr{R}(t)$ for $t>0$, we conclude that the set $\mathfrak{A}(t):=\{\mathscr{R}(t)z:\|z\|_{\mathbb{W}}\le 1\}$ is compact for $t>0$. Thus, there exists a finite family $\{\mathscr{R}(t)z_1,\ldots,\mathscr{R}(t)z_n\}\subset\mathfrak{A}(t)$ such that
		\begin{align}\label{e2.17}
			\|\mathscr{R}(t)z-\mathscr{R}(t)z_i\|_{\mathbb{W}}\le \frac{\epsilon}{3(K+1)},
		\end{align} 
		for any $z\in\mathbb{W}$ with $\|z\|_{\mathbb{W}}\le 1$. By the strong continuity of $\mathscr{R}(t)$, there exist constants $0<h_i<\min\{t,T\}$ for $i=1,\ldots,n,$ satisfying
		\begin{align}\label{e2.18}
			\|\mathscr{R}(t)z_i-\mathscr{R}(h)\mathscr{R}(t)z_i\|_{\mathbb{W}}\le \frac{\epsilon}{3},
		\end{align}
		for all $0\le h\le h_i$ and $i=1,\ldots,n$.
		Further, using the continuity in the uniform operator topology of $\mathscr{R}(t)$ for $t>0$, there exists a constant $0<\tilde{h}<\min\{h_1,...,h_n\}$ such that
		\begin{align}\label{e2.19}
			\|\mathscr{R}(t+h)z-\mathscr{R}(t)z\|_{\mathbb{W}}\le \frac{\epsilon}{3}, \ \mbox{for all}\ 0<h<\tilde{h}.
		\end{align} 
		Using \eqref{e2.17}, \eqref{e2.18} and \eqref{e2.19}, we estimate
		\begin{align*}
			\|\mathscr{R}(t+h)z-\mathscr{R}(h)\mathscr{R}(t)z\|_{\mathbb{W}}&\le\|\mathscr{R}(t+h)z-\mathscr{R}(t)z\|_{\mathbb{W}}+\|\mathscr{R}(t)z-\mathscr{R}(t)z_i\|_{\mathbb{W}}\\&\quad+\|\mathscr{R}(t)z_i-\mathscr{R}(h)\mathscr{R}(t)z_i\|_{\mathbb{W}}+\|\mathscr{R}(h)\mathscr{R}(t)z_i-\mathscr{R}(h)\mathscr{R}(t)z\|_{\mathbb{W}}\nonumber\\&\le\epsilon,
		\end{align*} 
		from which the claim (a) follows. We now prove the assertion (b). For this, let $t>0$ and $0<h<\tilde{h}$, and we estimate
		\begin{align*}
			\|\mathscr{R}(t)-\mathscr{R}(h)\mathscr{R}(t-h)\|_{\mathcal{L}(\mathbb{W})}&\le\|\mathscr{R}(t)-\mathscr{R}(t+h)\|_{\mathcal{L}(\mathbb{W})}+\|\mathscr{R}(t+h)-\mathscr{R}(h)\mathscr{R}(t)\|_{\mathcal{L}(\mathbb{W})}\\&\quad+\|\mathscr{R}(h)\mathscr{R}(t)-\mathscr{R}(h)\mathscr{R}(t-h)\|_{\mathcal{L}(\mathbb{W})}\\&\le\|\mathscr{R}(t)-\mathscr{R}(t+h)\|_{\mathcal{L}(\mathbb{W})}+\|\mathscr{R}(t+h)-\mathscr{R}(h)\mathscr{R}(t)\|_{\mathcal{L}(\mathbb{W})}\\&\quad+K\|\mathscr{R}(t)-\mathscr{R}(t-h)\|_{\mathcal{L}(\mathbb{W})}\\&\le \epsilon.
		\end{align*} 
		Thus, the claim (b) follows by the estimate (a) and the continuity of $\mathscr{R}(t)$ in the uniform operator topology for $t>0$.
	\end{proof}
	\begin{lem}\label{lem2.12}
		Let us assume that the resolvent operator $\mathscr{R}(t)$ is compact for $t>0$. Then the operator $\mathcal{I}:\mathrm{L}^{2}([a,b];\mathbb{W})\rightarrow C([a,b];\mathbb{W})$ is given by
		\begin{align*}
			(\mathcal{I}g)(t)= \int^{t}_{a}\mathscr{R}(t-s)g(s)\mathrm{d}s, \ t\in [a,b],
		\end{align*}
		is compact.
	\end{lem}
	\begin{proof} {\bf Step1  \kern-.2em(Equicontinuity):} \kern-.4em Consider the ball
			$\mathcal{B}_r\!=\!\!\left\{g\in \mathrm{L}^{2}([a,b];\mathbb{W})\!:\!\left\|g\right\|_{\mathrm{L}^{2}([a,b];\mathbb{W})}\leq r\right\},$
		for any $r>0$. For $ s_1,s_2\in [a,b]$  ($s_1<s_2$) and $g\in\mathcal{B}_{r}$, we compute
		\begin{align*}
			&\left\|(\mathcal{I}g)(s_2)-(\mathcal{I}g)(s_1)\right\|_{\mathbb{W}}\\&\le\left\|\int_{a}^{s_1}\left[\mathscr{R}(s_2-s)-\mathscr{R}(s_1-s)\right]g(s)\mathrm{d}s\right\|_{\mathbb{W}}+\left\|\int_{s_1}^{s_2}\mathscr{R}(s_2-s)g(s)\mathrm{d}s\right\|_{\mathbb{W}}\nonumber\\&\le\int_{a}^{s_1}\left\|\mathscr{R}(s_2-s)-\mathscr{R}(s_1-s)g(s)\right\|_{\mathbb{W}}\mathrm{d}s+K\int_{s_1}^{s_2}\|g(s)\|_{\mathbb{W}}\mathrm{d}s\nonumber\\&\le\int_{a}^{s_1}\left\|\mathscr{R}(s_2-s)-\mathscr{R}(s_1-s)g(s)\right\|_{\mathbb{W}}\mathrm{d}s+K\sqrt{(s_2-s_1)}\|g\|_{\mathrm{L}^2([a,b];\mathbb{W})}.
		\end{align*}
		If $s_1=a,$ then by the above inequality, we implies that 
		\begin{align*}
			\lim_{s_2\to a^+}\left\|(\mathcal{I}g)(s_2)-(\mathcal{I}g)(s_1)\right\|_{\mathbb{W}}=0,\; \mbox{unifromly for} \ g\in\mathrm{L}^2([a,b];\mathbb{W}).
		\end{align*} 
		For given $\epsilon>0$, let us take  $a+\epsilon<s_1<b$, we compute
		\begin{align}\label{e2.20}
		&	\left\|(\mathcal{I}g)(s_2)-(\mathcal{I}g)(s_1)\right\|_{\mathbb{W}}\nonumber\\&\le\int_{a}^{s_1}\left\|\mathscr{R}(s_2-s)-\mathscr{R}(s_1-s)\right\|_{\mathcal{L}(\mathbb{W})}\|g(s)\|_{\mathbb{W}}\mathrm{d}s+K\sqrt{(s_2-s_1)}\|g\|_{\mathrm{L}^2([a,b];\mathbb{W})}\nonumber\\&\le\int_{a}^{s_1-\epsilon}\left\|\mathscr{R}(s_2-s)-\mathscr{R}(s_1-s)\right\|_{\mathcal{L}(\mathbb{W})}\|g(s)\|_{\mathbb{W}}\mathrm{d}s\nonumber\\&\quad+\int_{s_1-\epsilon}^{s_1}\left\|\mathscr{R}(s_2-s)-\mathscr{R}(s_1-s)\right\|_{\mathcal{L}(\mathbb{W})}\|g(s)\|_{\mathbb{W}}\mathrm{d}s+K\sqrt{(s_2-s_1)}\|g\|_{\mathrm{L}^2([a,b];\mathbb{W})}\nonumber\\&\le\sup_{s\in[a,s_1-\epsilon]}\left\|\mathscr{R}(s_2-s)-\mathscr{R}(s_1-s)\right\|_{\mathcal{L}(\mathbb{W})}\int_{a}^{s_1-\epsilon}\|g(s)\|_{\mathbb{W}}\mathrm{d}s+2K\sqrt{\epsilon}\|g\|_{\mathrm{L}^2([a,b];\mathbb{W})}\nonumber\\&\quad+K\sqrt{(s_2-s_1)}\|g\|_{\mathrm{L}^2([a,b];\mathbb{W})}.
		\end{align}
		Using the continuity of $\mathscr{R}(t)$ for $t>0$ under the uniform operator topology and the arbitrariness of $\epsilon$, the right hand side of the expression \eqref{e2.20} converges to zero as $|s_2-s_1| \to 0$. Consequently, $\mathcal{I}(\mathcal{B}_r)$ is equicontinuous on $\mathrm{L}^{2}([a,b];\mathbb{W})$.
		
		\noindent {\bf Step 2:} Next, we verify that $\mathfrak{A}(t):= \left\{(\mathcal{I}g)(t):g\in \mathcal{B}_r\right\}$ for all $t\in [a,b]$ is relatively compact. At $t=a$, it is straightforward to verify that the set $\mathfrak{A}(t)$ is relatively compact in $\mathbb{W}$. Let $a<t\leq b$ be fixed, and for a given $\eta$ where $0< \eta<t-a$, we define
		\begin{align*}
			(\mathcal{I}^{\eta}g)(t)&=\mathscr{R}(\eta)\int_{a}^{t-\eta}\mathscr{R}(t-s-\eta)g(s)\mathrm{d}s.
		\end{align*}
		In view of the operator $\mathscr{R}(\eta)$ is compact, we see that the set $\mathfrak{A}_{\eta}(t)=\{(\mathcal{I}^{\eta}g)(t):g\in \mathcal{B}_r\}$ is relatively compact in $\mathbb{W}$. Therefore, there exist a finite $ z_{i}$'s, for $i=1,\dots, n$ in $ \mathbb{W} $ such that 
		\begin{align*}
			\mathfrak{A}_{\eta}(t) \subset \bigcup_{i=1}^{n}\mathcal{S}_{z_i}(\varepsilon/2),
		\end{align*}
		for some $\varepsilon>0$, where $\mathcal{S}_{z_i}(\varepsilon/2)$ is an open ball centered at $z_i$ and of radius $\varepsilon/2$. Using Lemma \ref{lem2.21}, we can choose an $\eta>0$ such that 
		\begin{align*}
			&\left\|(\mathcal{I}g)(t)-(\mathcal{I}^{\eta}g)(t)\right\|_{\mathbb{W}}\nonumber\\&\le\int_{a}^{t-\eta}\|\left[\mathscr{R}(t-s)-\mathscr{R}(\eta)\mathscr{R}(t-s-\eta)\right]g(s)\|_{\mathbb{W}}\mathrm{d}s+\int_{t-\eta}^{t}\|\mathscr{R}(t-s)g(s)\|_{\mathbb{W}}\mathrm{d}s\nonumber\\&\le\int_{a}^{t-2\eta}\|\left[\mathscr{R}(t-s)-\mathscr{R}(\eta)\mathscr{R}(t-s-\eta)\right]g(s)\|_{\mathbb{W}}\mathrm{d}s+K\sqrt{\eta}\|g\|_{\mathrm{L}^2([a,b];\mathbb{W})}\nonumber\\&\quad\int_{t-2\eta}^{t-\eta}\|\left[\mathscr{R}(t-s)-\mathscr{R}(\eta)\mathscr{R}(t-s-\eta)\right]g(s)\|_{\mathbb{W}}\mathrm{d}s\nonumber\\&\le\int_{a}^{t-2\eta}\left||\mathscr{R}(t-s)-\mathscr{R}(\eta)\mathscr{R}(t-s-\eta)\right\|_{\mathcal{L}(\mathbb{W})}\|g(s)\|_{\mathbb{W}}\mathrm{d}s+K\sqrt{\eta}\|g\|_{\mathrm{L}^2([a,b];\mathbb{W})}\nonumber\\&\quad+(K+K^2)\sqrt{\eta}\|g\|_{\mathrm{L}^2([a,b];\mathbb{W})}\le\frac{\varepsilon}{2}.
		\end{align*}
		Consequently $$ \mathfrak{A}(t)\subset \bigcup_{i=1}^{n}\mathcal{S}_{z_i}(\varepsilon).$$ Hence, for each $t\in [a,b]$, the set $\mathfrak{A}(t)$ is relatively compact in $ \mathbb{W}$. Applying the infinite-dimensional version of the Arzel\'a-Ascoli  theorem (refer to Theorem 3.7, Chapter 2, \cite{BI1995}), we can infer that the operator $\mathcal{I}$ is compact.
	\end{proof}
	\begin{lem}\label{lem5.3.2}Let us define the operator $\mathcal{G}:\mathrm{L}^1([a,b];\mathbb{W})\to C([a,b];\mathbb{W})$ such that 
		\begin{align}\label{eqn5.3.2}
			(\mathcal{G}g)(t):=\int_{a}^{t}\mathscr{R}(t-s)g(s)\mathrm{d}s,\ t\in [a,b].
		\end{align}
		If $\{g_n\}_{n=1}^{\infty}\subset\mathrm{L}^1([a,b];\mathbb{W})$ is any integrably bounded sequence, then the sequence $v_n:=\mathcal{G}(g_n)$ is relatively compact.\qed
	\end{lem}
	A proof of the lemma above can be readily obtained by using Lemma \ref{lem2.21} along with the technique employed in the proof of Lemma 3.6 of \cite{JDE2022}.
	\begin{cor}\label{cor1}
		Let $\{g_n\}_{n=1}^{\infty}\subset\mathrm{L}^1([a,b];\mathbb{W})$ be an integrably bounded sequence such that $g_n\xrightharpoonup{w} g \ \text{ in }\ \mathrm{L}^1([a,b];\mathbb{W}) \ \mbox{as}\ n\to\infty.$ Then
		$\mathcal{G}(g_n)\to\mathcal{G}(g)\ \text{ in }\ C([a,b];\mathbb{W}) \ \mbox{as}\ n\to\infty.$ \qed
	\end{cor}
	\subsection{Geometry of Banach spaces and duality mapping} In this subsection, we initially introduce some special geometry of Banach spaces and later we recall the notion of duality mapping and its important properties.
	\begin{Def} A Banach space $\mathbb{W}$ is said to be \emph{strictly convex} if for any $z_1,z_2\in\mathbb{W}$ with $\left\|z_1\right\|_{\mathbb{W}}=\left\|z_2\right\|_{\mathbb{W}}=1$ such that
		$\left\|\nu z_1+(1-\nu)z_2\right\|_{\mathbb{W}}<1,\ \mbox{for}\ 0<\nu<1.$
	\end{Def}
	\begin{Def} A Banach space $\mathbb{W}$ is said to be \emph{uniformly convex} if for any given $\epsilon>0,$ there exists a $\delta=\delta(\epsilon)$ such that
		$$\left\|z_1-z_2\right\|_{\mathbb{W}}\ge\epsilon \implies \left\|z_1+z_2\right\|_{\mathbb{W}}\le 2(1-\delta),$$
		where $z_1,z_2\in\mathbb{W}$ with $\left\|z_1\right\|_{\mathbb{W}}=1$ and $\left\|z_2\right\|_{\mathbb{W}}=1$.
	\end{Def}
\noindent 	Note that every \emph{uniformly convex} space $\mathbb{W}$ is also \emph{strictly convex}. The \emph{modulus of convexity}  of the Banach space $\mathbb{W}$ is defined as 
		\begin{align*}
			\delta_{\mathbb{W}}(\epsilon)=\inf\left\{1-\frac{\|z+y\|_{\mathbb{W}}}{2}:\|z\|_{\mathbb{W}}\le 1, \|y\|_{\mathbb{W}}\le 1, \|z-y\|_{\mathbb{W}}\ge\epsilon\right\}.
		\end{align*}
		The function $\delta_{\mathbb{W}}(\cdot)$ is defined on the interval $[0,2]$ is non-decreasing. In addition, a Banach space $\mathbb{W}$ is \emph{uniformly convex} if and only if $\delta_{\mathbb{W}}(\epsilon)>0$ for all $\epsilon>0$ (cf. \cite{IJM1975}).
		\begin{Def}
			A Banach space $\mathbb{W}$ is said to be \emph{uniformly smooth} if for any given $\epsilon>0$, there exists a $\delta>0$ such that for all $z_1,z_2\in\mathbb{W}$ with $\|z_1\|_{\mathbb{W}}=1$ and $\|z_2\|_{\mathbb{W}}\le\delta$, the inequality $$\frac{1}{2}\left(\|z_1+z_2\|_{\mathbb{W}}+\|z_1-z_2\|_{\mathbb{W}}\right)-1\le\epsilon\|z_2\|_{\mathbb{W}},$$ holds.
		\end{Def}
	\noindent 	The \emph{modulus of smoothness}  of the space $\mathbb{W}$ is given by
		\begin{align*}
			\rho_{_{\mathbb{W}}}(\tau)&=\sup\left\{\frac{\|z_1+\tau z_2\|_{\mathbb{W}}}{2}+\frac{\|z_1-\tau z_2\|_{\mathbb{W}}}{2}-1 : \|z_1\|_{\mathbb{W}}=, \|z_2\|_{\mathbb{W}}=1\right\}, \ \mbox{for all} \ \tau\in[0, \infty).
		\end{align*}
	\noindent 	Note that the space $\mathbb{W}$  is uniformly smooth if and only if (cf. \cite{IJM1975}) $\lim_{t\to 0^+}\frac{\rho_{_{\mathbb{W}}}(\tau)}{\tau}=0.$
	\begin{rem}\label{rem2.11}
		\begin{itemize}
			\item[(a)] A Banach space $\mathbb{W}$ is considered to be \emph{$p$-uniformly smooth}, where $1<p\le2$, if there exist an equivalent norm on $\mathbb{W}$ such that the modulus of smoothness $\rho_{_{\mathbb{W}}}(\tau)\le C\tau^p$ for some constant $C$. A Banach space is said to be \emph{$q$-uniformly convex}, where $2\le q< \infty$, if there exist an equivalent norm on $\mathbb{W}$ such that the modulus of convexity $\delta_{\mathbb{W}}(\epsilon)\ge C\epsilon^q$ for some constant $C>0$.
			\item[(b)] A space $\mathbb{W}$ is uniformly smooth if and only if its dual space $\mathbb{W}^*$ is uniformly convex. A space $\mathbb{W}$ is uniformly convex if and only if $\mathbb{W}^*$ is uniformly smooth (cf. \cite{SP2006}). Moreover, every uniformly convex or uniformly smooth space is $q$-uniformly convex and $p$-uniformly smooth for some $q<\infty$ and $p>1$ (see \cite{IJM1975}).
		\end{itemize}
	\end{rem}
	\begin{Ex}
		Here are some example of smooth and convex spaces (cf. \cite{PTRS2010,JMAA1991}).
		\begin{itemize}
			\item [(i)] Every Hilbert space is $2$-uniformly smooth and $2$-uniformly convex.
			\item [(ii)] Let $\Omega\subset\mathbb{R}^n$ be a measurable set. The space $\mathrm{L}^p(\Omega)$ for $2\le p<\infty$ is $2$-uniformly smooth and  $p$-uniformly convex.
			\item [(iii)] The Sobolev space $\mathrm{W}^{m,p}(\Omega)$ for $2\le p<\infty$ and $m\in\mathbb{N}$ is $2$-uniformly smooth and  $p$-uniformly convex.
		\end{itemize}
	\end{Ex}
	Let us recall the notion of the duality mapping and its property.
	\begin{Def}\label{Def2.5.1}
		A multi-valued map $\mathscr{J}:\mathbb{W}  \rightarrow 2^{\mathbb{W}^*}$  given by $$\mathscr{J}[z]=\{z^*\in \mathbb{W}^*:\langle z, z^*\rangle=\|z\|_{\mathbb{W}}^{2}=\|z^*\|_{\mathbb{W}^*}^{2}\}, \ \text{ for all } \ z\in \mathbb{W}.$$  is called the \emph{duality mapping}.
	\end{Def}
	Note that $\mathscr{J}[\lambda z]= \lambda\mathscr{J}[z]$, for all $\lambda\in\mathbb{R}$ and $z\in\mathbb{W}$.
	\begin{rem}\label{rem2.5.1}
		\begin{itemize}
			\item[(i)]
			If $\mathbb{W}$ is a reflexive Banach space with the norm $\left\|\cdot\right\|_{\mathbb{W}}$, then it can always be renormed such that both $\mathbb{W}$ and $\mathbb{W}^*$ becomes strictly convex  (cf. \cite{IJM1967}). According to Milman theorem (cf. \cite{Sp1978}), every uniformly convex Banach space is reflexive.
			\item [(ii)] The strict convexity of $\mathbb{W}^*$ ensures that the mapping $\mathscr{J}$ is single valued and demicontinuous (cf. \cite{AP1992}), that is,
			$$z_{k} \rightarrow z \ \text{in}\ \mathbb{W} \implies\mathscr{J}[z_{k}]\xrightharpoonup{{w}}\mathscr{J}[z] \ \text{in}\ \mathbb{W}^*\ \text{as}\ k \rightarrow \infty.$$
		\end{itemize}
	\end{rem}
	\subsection{Phase space} We present definition of the phase space $\mathfrak{B}$, as introduced in the profound work of Hale and Kato (cf. \cite{SPV1991}). Specifically, $\mathfrak{B}$ denotes a linear space comprising all mappings from $(-\infty, 0]$ into $\mathbb{W}$ equipped with the seminorm $\left\|\cdot\right\|_{\mathfrak{B}}$ and satisfying the following axioms:
	\begin{enumerate}
		\item [(A1)] Let $z: (-\infty, \sigma+\vartheta)\rightarrow \mathbb{W},\; \vartheta>0$ be a continuous function on $[\sigma, \sigma+\vartheta)$ and $z_{\sigma}\in \mathfrak{B}$. Then for every $t\in [\sigma, \sigma+\vartheta),$  the following conditions hold:
		\begin{itemize}
			\item [(i)] $z_{t}$ is in $\mathfrak{B}.$
			\item [(ii)] $\left\|z(t)\right\|_{\mathbb{W}}\leq K_{1}\left\|z_{t}\right\|_{\mathfrak{B}}$, where $K_{1}$ is a constant independent from $z(\cdot)$.
			\item [(iii)] $\left\|z_{t}\right\|_{\mathfrak{B}}\leq \varLambda(t-\sigma)\sup\{\left\|z(s)\right\|_{\mathbb{W}}: \sigma 
			\leq s\leq t\}+\varUpsilon(t-\sigma)\left\|z_{\sigma}\right\|_{\mathfrak{B}},$ where $\varLambda:[0, \infty)\rightarrow [1, \infty)$ is  continuous and $\varUpsilon:[0, \infty)\rightarrow [1, \infty)$ is locally bounded, and both $\varLambda,\varUpsilon$ are independent of $z(\cdot).$
		\end{itemize}
		\item [(A2)]For a given $z(\cdot)$ in (A1), the mapping $t\to z_t$ is continuous from $[\sigma, \sigma+\vartheta)$ into $\mathfrak{B}$.
		\item [(A3)] The space $\mathfrak{B}$ is complete. 
	\end{enumerate}  
	\begin{Ex}
		Let $\omega:(-\infty, -r]\to\mathbb{R}^+$ be a Lebesgue integrable function. Take $\mathfrak{B}=C_{r}\times\mathrm{L}^1_\omega(\mathbb{W})$ as the space of all mapping $\phi:(-\infty,0]\to\mathbb{W}$ such that $\phi\vert_{[-r,0]}\in C([-r,0];\mathbb{W}),$ for some $r>0$, $\phi$ is Lebesgue measurable on $(-\infty,-r)$, and $\omega\|\phi(\cdot)\|_{\mathbb{W}}$ is Lebesgue integrable on $(-\infty,-r]$. The seminorm in $\mathfrak{B}$ is given as 
		\begin{align*}
			\label{Bnorm}\left\|\phi\right\|_{\mathfrak{B}}:=\sup\{\|\phi(\theta)\|_\mathbb{W}: -r\le\theta\le0\} +\int_{-\infty}^{-r}\omega(\theta)\|\phi(\theta)\|_{\mathbb{W}}\mathrm{d}\theta.
		\end{align*} Furthermore, there exists a locally bounded function $\mathcal{E}:(-\infty,0]\to\mathbb{R}^+$ such that $\omega(t+\theta)\le \mathcal{E}(t)\omega(\theta),$ for all $t\le0$ and $\theta\in(-\infty,0)\backslash \varOmega_t$ where $\varOmega_t\subseteq(-\infty,0)$ is a set with Lebesgue measure zero. A simple example of $\omega$ is given by $\omega(\theta)=e^{\nu\theta}$ for some $\nu>0$. 
	\end{Ex}
	The space $\mathfrak{B}=C_{r}\times\mathrm{L}^1_\omega(\mathbb{W})$ verifies the conditions (A1)-(A3) for $t\ge0$ with $K_1=1,$
	\begin{equation*}
		\varLambda(t)=	\left\{
		\begin{aligned}
			1, \qquad\qquad\qquad\  &\mbox{for}\ 0\le t\le r,\\
			1+\int_{-t}^{-r}\omega(\theta)\mathrm{d}\theta,\ &\mbox{for}\ r<t
		\end{aligned}
		\right.
	\end{equation*}
	and 
	\begin{equation*}
		\varUpsilon(t)=	\left\{
		\begin{aligned}
			&\max\left\{1+\int_{-r-t}^{-r}\omega(\theta)\mathrm{d}\theta, \mathcal{E}(-t)\right\}, \ \mbox{for}\ 0\le t\le r,\\
			&\max\left\{\int_{-r-t}^{-r}\omega(\theta)\mathrm{d}\theta, \mathcal{E}(-t)\right\}, \quad\quad \mbox{for}\ r<t,
		\end{aligned}
		\right.
	\end{equation*}
	see Theorem 3.18, \cite{SPV1991}.
	\subsection{Mild solution and controllability operators} This subsection starts with the definition of a mild solution for the semilinear system \eqref{SEq}. Later, we introduce the controllability operators. Given $\psi\in\mathfrak{B}$ fixed, define  $f_i:J\to\mathbb{W}, i=1,2,$ as $$f_1(t)=-\int_{-\infty}^{0}\mathrm{G}(t-s)\psi(s)\mathrm{d}s \ \text{and} \ f_2(t)=\int_{-\infty}^{0} \mathrm{N}(t-s)\psi(s)ds.$$
	\begin{Def}[Mild solution]\label{Def2.19}
		A function $w:(-\infty,T]\to\mathbb{W}$ is called a \emph{mild solution} of the neutral system \eqref{SEq}, if $w_0=\psi\in\mathfrak{B}$, $f_1$ is differentiable with $f'_1\in\mathrm{L}^1(J;\mathbb{W}), w|_{J}\in C(J;\mathbb{W})$ and verify the following integral equation:
		\begin{align}\label{Md}
			w(t)&=\mathscr{R}(t)\psi(0)+\int_{0}^{t}\mathscr{R}(t-s)\left[\mathrm{B}u(s)+f(s,w_s)\right]\mathrm{d}s\nonumber\\&\qquad+\int_{0}^{t}\mathscr{R}(t-s)\left[f'_1(s)+f_2(s)\right]\mathrm{d}s, t\in J.
		\end{align}
	\end{Def}
	\begin{Def}
		The system \eqref{SEq} is said to be \emph{approximately controllable} over $J$ if, for any given function  $\psi\in\mathfrak{B}$ and any $\zeta_1\in\mathbb{W}$, and every $\epsilon>0$, there exist a control function $u\in\mathrm{L}^{2}(J;\mathbb{U})$ such that the mild solution $w(\cdot)$ of the problem \eqref{SEq} satisfy the following:
		$$\left\|w(T)-\zeta_1\right\|_{\mathbb{W}}\le\epsilon.$$ 
	\end{Def}
\noindent 	In the remaining part  of this article, we assume that $\mathbb{W}$ is a separable reflexive Banach space having uniformly convex dual, unless otherwise specified.
	Next, we define the operators 
	\begin{equation}\label{Copt}
		\left\{
		\begin{aligned}
		&	\mathrm{L}_Tu:=\int_0^T\mathscr{R}(T-s)\mathrm{B}u(t)\mathrm{d}t\\
		&	\Psi_{0}^{T}:=\mathrm{L}_T\mathrm{L}_T^*=\int^{T}_{0}\mathscr{R}(T-t)\mathrm{B}\mathrm{B}^{*}\mathscr{R}(T-t)^{*}\mathrm{d}t \\
		&	\mathrm{R}(\lambda,\Psi_{0}^{T}) :=(\lambda \mathrm{I}+\Psi_{0}^{T}\mathscr{J})^{-1},\ \lambda > 0.
		\end{aligned}
		\right.
	\end{equation}
	From the second expression of \eqref{Copt}, it is evident that the operator $\Psi_{0}^{T}:\mathbb{W}^*\to\mathbb{W}$ is  nonnegative and symmetric. It is clear from the third expression of \eqref{Copt} that the operator $\mathrm{R}(\lambda,\Psi_{0}^{T})$ is nonlinear. Further, note that if $\mathbb{W}$ is a separable Hilbert space identified by its own dual, then the duality mapping $\mathscr{J}$ becomes $\mathrm{I}$ (the identity operator), thus $\mathrm{R}(\lambda,\Psi_{0}^{T}):=(\lambda \mathrm{I}+\Psi_{0}^{T})^{-1},\ \lambda > 0$ is a linear operator.
	\begin{lem}\label{lem2.16}
		For $y\in\mathbb{W}$ and $\lambda>0$, the equation 
		\begin{align}\label{e2.7}
			\lambda z_\lambda+\Psi_{0}^{T}\mathscr{J}[z_\lambda]=\lambda y,
		\end{align}
		has a unique solution  $$z_\lambda=z_{\lambda}(y)=\lambda(\lambda \mathrm{I}+\Psi_{0}^{T}\mathscr{J})^{-1}(y).$$ Moreover
		\begin{align*}
			\left\|z_{\lambda}(y)\right\|_{\mathbb{W}}=\left\|\mathscr{J}[z_{\lambda}(y)]\right\|_{\mathbb{W}^*}\leq\left\|y\right\|_{\mathbb{W}}.
		\end{align*}
	\end{lem}
	A proof of the aforementioned lemma can be derived in a manner similar to that used in the proof of Lemma 2.2 in \cite{SIAM2003}.
	\begin{lem}\label{lem2.17}
		The operator $\mathrm{R}(\lambda,\Psi_{0}^{T}):\mathbb{W}\to\mathbb{W},\ \lambda > 0$  is uniformly continuous in every bounded subset of $\mathbb{W}$.
	\end{lem}
	\begin{proof}
		Let us define a set
		$$B_r:=\{w\in\mathbb{W}:\|w\|_{\mathbb{W}}\le r\},$$
		where $r$ is a positive constant.
		Now, for any $w_1,w_2\in B_r$, let us estimate
		\begin{align}\label{e2.8}
			&\left\langle (\lambda \mathrm{I}+\Psi_{0}^{T}\mathscr{J})(w_2)-(\lambda \mathrm{I}+\Psi_{0}^{T}\mathscr{J})(w_1), \mathscr{J}[w_2]-\mathscr{J}[w_1]\right\rangle\nonumber\\&=\Big\langle \lambda(w_2-w_1), \mathscr{J}[w_2]-\mathscr{J}[w_1]\Big\rangle+\left\langle \Psi_0^T(\mathscr{J}[w_2]-\mathscr{J}[w_1]), \mathscr{J}[w_2]-\mathscr{J}[w_1]\right\rangle\nonumber\\&=\lambda\left\langle w_2-w_1, \mathscr{J}[w_2]-\mathscr{J}[w_1]\right\rangle+\|\mathrm (\mathrm{L}_T)^*(\mathscr{J}[w_2]-\mathscr{J}[w_1])\|_{\mathbb{U}}^2\nonumber\\&\ge\lambda\left\langle w_2-w_1,\mathscr{J}[w_2]-\mathscr{J}[w_1]\right\rangle.
		\end{align}
		Since the space $\mathbb{W}^*$ is uniformly convex,  by Remark \ref{rem2.11}, the space $\mathbb{W}$ is uniformly smooth. Moreover, the space $\mathbb{W}$ is $q$-uniformly convex and $p$-uniformly smooth for some $2\le q<\infty$ and $1<p\le2$. By applying Theorem 1.6.4, \cite{SP2006} and the definition of $q$-uniformly convexity, we obtain
		\begin{align}\label{e2.9}
			\left\langle w_2-w_1,\mathscr{J}[w_2]-\mathscr{J}[w_1]\right\rangle\ge \frac{C}{2Lc_r^q}\|w_2-w_1\|_{\mathbb{W}}^q,
		\end{align}
		where $1<L<1.7$ is the Figiel constant, $c_r=2\max\{1,r\}$ and $C>0$. Combining \eqref{e2.8} and \eqref{e2.9}, we have the following:
		\begin{align}\label{e2.10}
			&\left\langle (\lambda \mathrm{I}+\Psi_{0}^{T}\mathscr{J})(w_2)-(\lambda \mathrm{I}+\Psi_{0}^{T}\mathscr{J})(w_1), \mathscr{J}[w_2]-\mathscr{J}[w_1]\right\rangle\ge \frac{\lambda C}{2Lc_r}\|w_2-w_1\|_{\mathbb{W}}^q.
		\end{align}
		Using  the Cauchy-Schwarz inequality, we compute
		\begin{align}\label{e2.11}
			&\left\langle (\lambda \mathrm{I}+\Psi_{0}^{T}\mathscr{J})(w_2)-(\lambda \mathrm{I}+\Psi_{0}^{T}\mathscr{J})(w_1), \mathscr{J}[w_2]-\mathscr{J}[w_1]\right\rangle\nonumber\\&\le \|(\lambda \mathrm{I}+\Psi_{0}^{T}\mathscr{J})(w_2)-(\lambda \mathrm{I}+\Psi_{0}^{T}\mathscr{J})(w_1)\|_{\mathbb{W}}\|\mathscr{J}[w_2]-\mathscr{J}[w_1]\|_{\mathbb{W}^*}.
		\end{align} 
		Note that the space $\mathbb{W}$ is $p$-uniformly smooth implies that the dual space $\mathbb{W}^*$ is $\frac{p}{p-1}$-uniformly convex. Using Corollary 1.6.7, \cite{SP2006}, we obtain
		\begin{align}\label{e2.12}
			\left\langle \mathscr{J}[w_2]-\mathscr{J}[w_1], w_2-w_1\right\rangle\ge \frac{\bar{C}}{2Lc_r^{\frac{p}{p-1}}}\|\mathscr{J}[w_2]-\mathscr{J}[w_1]\|^{\frac{p}{p-1}}_{\mathbb{W}^*},
		\end{align}
		where $\bar{C}$ is a positive constant. Once again using the Cauchy-Schwarz inequality, we estimate
		\begin{align}\label{e2.13}
			\left\langle \mathscr{J}[w_2]-\mathscr{J}[w_1], w_2-w_1\right\rangle\le \|\mathscr{J}[w_2]-\mathscr{J}[w_1]\|_{\mathbb{W}^*}\|w_2-w_1\|_{\mathbb{W}}.	
		\end{align}
		From the estimates \eqref{e2.12} and \eqref{e2.13}, we get
		\begin{align}\label{e2.14}
			\|\mathscr{J}[w_2]-\mathscr{J}[w_1]\|_{\mathbb{W}^*}\le \frac{(2L)^{p-1}c_r^p}{\bar{C}} \|w_2-w_1\|_{\mathbb{W}}^{p-1}.
		\end{align}
		Combining \eqref{e2.11} and \eqref{e2.14}, we compute
		\begin{align*}
			&\left\langle (\lambda \mathrm{I}+\Psi_{0}^{T}\mathscr{J})(w_2)-(\lambda \mathrm{I}+\Psi_{0}^{T}\mathscr{J})(w_1), \mathscr{J}[w_2]-\mathscr{J}[w_1]\right\rangle\nonumber\\&\le \frac{(2L)^{p-1}c_R^p}{\bar{C}}\|(\lambda \mathrm{I}+\Psi_{0}^{T}\mathscr{J})(w_2)-(\lambda \mathrm{I}+\Psi_{0}^{T}\mathscr{J})(w_1)\|_{\mathbb{W}} \|w_2-w_1\|_{\mathbb{W}}^{p-1}.
		\end{align*}
		Further from the estimates \eqref{e2.10} and \eqref{e2.14}, we have
		\begin{align}\label{e2.15}
			\|(\lambda \mathrm{I}+\Psi_{0}^{T}\mathscr{J})(w_2)-(\lambda \mathrm{I}+\Psi_{0}^{T}\mathscr{J})(w_1)\|_{\mathbb{W}}\ge C_{R} \|w_2-w_1\|_{\mathbb{W}}^s,
		\end{align}
		where $C_{r}=\frac{\lambda C}{\bar{C}(2L)^pc_r^{p+1}}$ and $s=q-p+1\in[1,\infty)$. Since the map $(\lambda\mathrm{I}+\Psi_{0}^{T}\mathscr{J})$  is invertible,  there exist $x_1,x_2\in\mathbb{W}$ such that $$(\lambda \mathrm{I}+\Psi_{0}^{T}\mathscr{J})^{-1}(x_2)=w_2, \  (\lambda \mathrm{I}+\Psi_{0}^{T}\mathscr{J})^{-1}(x_1)=w_1.$$ This implies $$(\lambda\mathrm{I}+\Psi_{0}^{T}\mathscr{J})(w_2)=x_2, \  (\lambda \mathrm{I}+\Psi_{0}^{T}\mathscr{J})(w_1)=x_1.$$
		Finally, from the estimate \eqref{e2.15}, we obtain
		\begin{align*}
			\|(\lambda \mathrm{I}+\Psi_{0}^{T}\mathscr{J})^{-1}(x_2)-(\lambda \mathrm{I}+\Psi_{0}^{T}\mathscr{J})^{-1}(x_1)\|_{\mathbb{W}}\le \tilde{C_r} \|x_2-x_1\|_{\mathbb{W}}^{\theta}
		\end{align*}
		where $\tilde{C_r}=\frac{1}{C_r}>0$ and $\theta=\frac{1}{s}\ (0<\theta\le 1)$. Thus, the system is uniformly continuous on every bounded subset of the space $\mathbb{W}$.
	\end{proof}
	To ensure the existence of a mild solution and the approximate controllability of the system \eqref{SEq}, we impose the following assumptions.
	\begin{Ass}\label{ass2.4} Let us take the following assertions:
		\begin{enumerate}
			\item [\textit{(H0)}] The family $z_{\lambda}=z_{\lambda}(y)=\lambda\mathrm{R}(\lambda,\Psi_{0}^{T})(y) \rightarrow 0$ as $\lambda\downarrow 0$ for every $y\in \mathbb{W}$, under the strong topology, where $z_{\lambda}(y)$ represents a solution of the equation \eqref{e2.7}.
			\item[\textit{(H1)}] Let $\mathrm{R}(\lambda_0,\mathrm{A})$ be compact for some $\lambda_0\in\rho(\mathrm{A})$.
			\item [\textit{(H2)}]   The function $f(t,\cdot): \mathfrak{B}\rightarrow \mathbb{W}$ is continuous for a.e. $t\in J$. The map $t\mapsto f(t, \phi) $ is strongly measurable in $J$ for each $ \phi\in \mathfrak{B}$. Also, there exists a function $\gamma\in\mathrm{L}^1(J;\mathbb{R}^+)$ such that 
			$$\|f(t,\phi)\|_{\mathbb{W}}\le\gamma(t), \ \mbox{for a.e.}\ t\in J \ \mbox{and for all}\ \phi\in\mathfrak{B}.$$
		\end{enumerate}
	\end{Ass}
	\section{Linear Control system}\label{LCS}\setcounter{equation}{0} 
	The following section is dedicated to explore the approximate controllability of a linear control system corresponding to \eqref{SEq}. In this section, we first formulate an optimal control problem and subsequently discuss its relationship to the approximate controllability of the linear control system. We first define the \emph{mild} and \emph{classical} solutions of the following non-homogeneous linear system:
	\begin{equation}\label{NLEq}
		\left\{
		\begin{aligned}
			\frac{\mathrm{d}}{\mathrm{d}t}\left[w(t)+\int_{0}^{t}\mathrm{G}(t-s)w(s)\mathrm{d}s\right]&=\mathrm{A}w(t)+\int_{0}^{t}\mathrm{N}(t-s)w(s)\mathrm{d}s+g(t),\ t\in(0,T],\\
			w(0)&=\zeta\in\mathbb{W}.
		\end{aligned}
		\right.
	\end{equation}
	If the function $g\in\mathrm{L}^1(J;\mathbb{W})$, then a function $w\in C(J;\mathbb{W})$ given by the  expression
	\begin{align}\label{3.2}
		w(t)&=\mathscr{R}(t)\zeta+\int_{0}^{t}\mathscr{R}(t-s)g(s), t\in J,
	\end{align}
	is called a \emph{mild solution} of the system \eqref{NLEq}. Moreover, if $\zeta\in D(\mathrm{A})$ and $g\in C(J;D_{1}(\mathrm{A}))$, then the function $w(\cdot)$ given in \eqref{3.2} is a \emph{classical solution}. The work in \cite{JIEA2011} thoroughly examines the existence of both mild and classical solutions for the non-homogeneous linear system mentioned above.
	\subsection{Optimal control problem and approximate controllability} We now formulate an optimal control problem to produce the approximate controllability of the linear system. For this, we consider a linear regulator problem, which involves minimizing a cost functional defined as follows
	\begin{equation}\label{CF}
		\mathscr{F}(w,u)=\left\|w(T)-\zeta_1\right\|^{2}_{\mathbb{W}}+\lambda\int^{T}_{0}\left\|u(t)\right\|^{2}_{\mathbb{U}}\mathrm{d}t,
	\end{equation}
	where $\lambda >0, \zeta_1\in\mathbb{W}$ and $w(\cdot)$ is a mild solution of the system
	\begin{equation}\label{LEq1}
		\left\{
		\begin{aligned}
			\frac{\mathrm{d}}{\mathrm{d}t}\left[w(t)+\int_{0}^{t}\mathrm{G}(t-s)w(s)\mathrm{d}s\right]&=\mathrm{A}w(t)+\int_{0}^{t}\mathrm{N}(t-s)w(s)\mathrm{d}s+\mathrm{B}u(t),\ t\in(0,T],\\
			w(0)&=\zeta\in\mathbb{W},
		\end{aligned}
		\right.
	\end{equation}
	with control $u(\cdot)\in \mathbb{U}$. Since $\mathrm{B}u\in\mathrm{L}^1(J;\mathbb{W})$, the system \eqref{LEq1} has a unique mild solution $w\in \mathrm{C}(J;\mathbb{W}) $ given by
	\begin{align*}
		w(t)= \mathscr{R}(t)\zeta+\int^{t}_{0}\mathscr{R}(t-s)\mathrm{B}u(s)\mathrm{d}s,\ t\in J,
	\end{align*}
	for any $u\in\mathcal{U}_{\mathrm{ad}}=\mathrm{L}^2(J;\mathbb{U})$ (class of \emph{admissible controls}). The \emph{admissible class} is given by  $$\mathcal{A}_{\mathrm{ad}}:=\big\{(w,u) :w\ \text{is a unique mild solution of}\ \eqref{LEq1} \ \text{associated with the control}\ u\in\mathcal{U}_{\mathrm{ad}}\big\}.$$ Since for any $u \in \mathcal{U}_{\mathrm{ad}}$, there is a unique mild solution of the system \eqref{LEq1}. Hence, the set $\mathcal{A}_{\mathrm{ad}}$ is nonempty. Now, the optimal control problem is to find $(w^0,u^0)\in\mathcal{A}_{\mathrm{ad}}$ such that 
	\begin{align}\label{opt}
	 \mathscr{F}(w^0,u^0) =	\min_{ (w,u) \in \mathcal{A}_{\mathrm{ad}}}  \mathscr{F}(w,u).
	\end{align}
	The following theorem determines the existence of an optimal pair for the problem outlined in \eqref{opt}. In order to prove the following theorem we will use the similar technique as discussed in \cite{MTM-20}.
	\begin{theorem}\label{exi}
		For any  $\zeta\in\mathbb{W}$, the minimization problem \eqref{opt} possesses a unique optimal pair $(w^0,u^0)\in\mathcal{A}_{\mathrm{ad}}$.
	\end{theorem}
	\begin{proof}
		Denote  $L := 	\min\limits_{ (w,u) \in \mathcal{A}_{\mathrm{ad}}}  \mathscr{F}(w,u) =\inf \limits _{u \in \mathcal{U}_{\mathrm{ad}}}\mathscr{F}(w,u).$
		Since, $0\leq L < +\infty$, we can find a minimizing sequence $\{u^n\}_{n=1}^{\infty} \subset \mathcal{U}_{\mathrm{ad}}$ satisfying $\lim\limits_{n\to\infty}\mathscr{F}(w^n,u^n) = L,$ where $(w^n, u^n)\in\mathcal{A}_{\mathrm{ad}}$ for each $n\in\mathbb{N}$. Evidently, the function  $w^n(\cdot)$ is given by
		\begin{align}\label{eq3.6}
			w^n(t)= \mathscr{R}(t)\zeta+\int^{t}_{0}\mathscr{R}(t-s)\mathrm{B}u^n(s)\mathrm{d}s,\ t\in J.
		\end{align}
		Since  $0\in\mathcal{U}_{\mathrm{ad}}$, without loss of generality, we may assume that $\mathscr{F}(w^n,u^n) \leq \mathscr{F}(w,0)$, where $(w,0)\in\mathcal{A}_{\mathrm{ad}}$, implies that
		\begin{align*}
			\left\|w^n(T)-\zeta_1\right\|^{2}_{\mathbb{W}}+\lambda\int^{T}_{0}\left\|u^n(t)\right\|^{2}_{\mathbb{U}}\mathrm{d}t\leq \left\|w(T)-\zeta_1\right\|^{2}_{\mathbb{W}}\leq 2\left(\|w(T)\|_{\mathbb{W}}^2+\|\zeta_1\|_{\mathbb{W}}^2\right)<+\infty.
		\end{align*}
		The above estimate ensures that, there exists a $\widetilde{L}>0$  large enough such that
		\begin{align}\label{e3.7}\int_0^T \|u^n(t)\|^2_{\mathbb{U}} \mathrm{d} t \leq  \widetilde{L} < +\infty 
			.\end{align} 
		Using \eqref{eq3.6}, we estimate
		\begin{align*}
			\|w^n(t)\|_{\mathbb{W}}&\leq\|\mathscr{R}(t)\zeta\|_{\mathbb{W}}+\int_0^t\|\mathscr{R}(t-s)\mathrm{B}u^n(s)\|_{\mathbb{W}}\mathrm{d}s \nonumber\\&\leq\|\mathscr{R}(t)\|_{\mathcal{L}(\mathbb{W})}\|\zeta\|_{\mathbb{W}}+\int_0^t\|\mathscr{R}(t-s)\|_{\mathcal{L}(\mathbb{W})}\|\mathrm{B}\|_{\mathcal{L}(\mathbb{U};\mathbb{W})}\|u^n(s)\|_{\mathbb{U}}\mathrm{d}s\nonumber\\&\leq K\|\zeta\|_{\mathbb{W}}+KMt^{1/2}\left(\int_0^t\|u^n(s)\|_{\mathbb{U}}^2\mathrm{d}s\right)^{1/2}\nonumber\\&\leq K\|\zeta\|_{\mathbb{W}}+KMt^{1/2}\widetilde{L}^{1/2}<+\infty,
		\end{align*}
		for all $t\in J$. Since  $\mathrm{L}^{2}(J;\mathbb{W})$ is reflexive,  by  Banach-Alaoglu theorem, there exists a subsequence $\{w^{n_j}\}_{j=1}^{\infty}$  such that $
			w^{n_j}\xrightharpoonup{w}w^0\ \text{ in }\mathrm{L}^{2}(J;\mathbb{W})$ as $ j\to\infty. $
		It is clear from the relation \eqref{e3.7} the sequence $\{u^n\}_{n=1}^{\infty}$ is uniformly bounded in the space $\mathrm{L}^2(J;\mathbb{U})$. Again by Banach-Alaoglu theorem, we can extract a subsequence $\{u^{n_j}\}_{j=1}^{\infty}$ such that 
	$u^{n_j}\xrightharpoonup{w}u^0\ \text{ in }\mathrm{L}^2(J;\mathbb{U})$ as $ j\to\infty. 
	$
		We also have
		\begin{align}\label{e3.8}
			\mathrm{B}	u^{n_j}\xrightharpoonup{w}\mathrm{B}u^0\ \text{ in }\mathrm{L}^2(J;\mathbb{W})\ \text{ as } \ j\to\infty. 
		\end{align}
		The weak convergence given in \eqref{e3.8} and the compactness of the operator (Lemma \ref{lem2.12}) $$(Qg)(\cdot)= \int_{0}^{\cdot}\mathscr{R}(\cdot-s)g(s)\mathrm{d}s : \mathrm{L}^{2}(J;\mathbb{W}) \rightarrow \mathrm{C}(J;\mathbb{W})$$
		it follows that 
		\begin{align*}
			&	\left\|\int_0^t\mathscr{R}(t-s)\mathrm{B}u^{n_j}(s)\mathrm{d}s-\int_0^t\mathscr{R}(t-s)\mathrm{B}u^0(s)\mathrm{d}s\right\|_{\mathbb{W}}\to 0\ \text{ as }\ j\to\infty, 
		\end{align*}
		for all $t\in J$. Moreover, for all $t\in J, $ we obtain
		\begin{align*}
			\|w^{n_j}(t)-w^0(t)\|_{\mathbb{W}}=	\left\|\int_0^t\mathscr{R}(t-s)\mathrm{B}u^{n_j}(s)\mathrm{d}s-\int_0^t\mathscr{R}(t-s)\mathrm{B}u^0(s)\mathrm{d}s\right\|_{\mathbb{W}}\nonumber\to 0
		\end{align*}
	as $j\to\infty$, 	where 
		\begin{align*}
			w^{0}(t)=\mathscr{R}(t)\zeta+\int^{t}_{0}\mathscr{R}(t-s)\mathrm{B}u^0(s)\mathrm{d}s, \ t\in J.
		\end{align*}
		It is clear from the above formulation that $w^0(\cdot)$ is the unique mild solution of the system \eqref{LEq1} with the control $u^0$. The continuity  of $w^{n_j}(\cdot)$ in $\mathbb{W}$ implies that  $w^{n_j}\to w^0$ in $\mathrm{C}(J;\mathbb{W})$ as $j\to\infty$. Since $w^0(\cdot)$ is a unique mild solution of \eqref{LEq1},  the whole sequence $\{w^n\}_{n=1}^{\infty}$ converges to  $w^0$. The fact $u^0\in\mathcal{U}_{\mathrm{ad}}$ immediately infer that $(w^0,u^0)\in\mathcal{A}_{\mathrm{ad}}$.
		
		Finally, we verify that  $(w^0,u^0)$ is a minimizer, that is, \emph{$L=\mathscr{F}(w^0,u^0)$}. Note that the cost functional $\mathscr{F}(\cdot,\cdot)$ given in \eqref{CF} is convex and continuous on $\mathrm{L}^2(J;\mathbb{W}) \times \mathrm{L}^2(J;\mathbb{U})$ (cf. Proposition III.1.6 and III.1.10,  \cite{CP1983}). This ensures that the functional  $\mathscr{F}(\cdot,\cdot)$ is sequentially weakly lower semi-continuous (cf. Proposition II.4.5, \cite{CP1983}). That is, for a sequence 
		$$(w^n,u^n)\xrightharpoonup{w}(w^0,u^0)\ \text{ in }\ \mathrm{L}^2(J;\mathbb{W}) \times  \mathrm{L}^2(J;\mathbb{U}),$$
		we have 
		\begin{align*}
			\mathscr{F}(w^0,u^0) \leq  \liminf \limits _{n\rightarrow \infty} \mathscr{F}(w^n,u^n).
		\end{align*}
		Consequently, we obtain 
		\begin{align*}L \leq \mathscr{F}(w^0,u^0) \leq  \liminf \limits _{n\rightarrow \infty} \mathscr{F}(w^n,u^n)=  \lim \limits _{n\rightarrow \infty} \mathscr{F}(w^n,u^n) = L,\end{align*}
		and thus $(w^0,u^0)$ is a minimizer of the problem \eqref{opt}.
		
		The uniqueness of the optimal pair $(w^0,u^0)$ is followed by the facts the cost functional defined in \eqref{CF} is convex, the constraint \eqref{LEq1} is linear and $\mathcal{U}_{\mathrm{ad}}=\mathrm{L}^2(J;\mathbb{U})$ is convex. Hence proof of Theorem 3.1 is complete.
	\end{proof}
	Before discussing the explicit expression for the optimal control ${u}$, let us first investigate the differentiability of the mapping $w\mapsto\frac{1}{2}\|w\|_{\mathbb{W}}^2$. Consider the function $k: \mathbb{W}\to\mathbb{R}$ defined as $k(w)=\frac{1}{2}\|w\|_{\mathbb{W}}^2$.  Since $\mathbb{W}$ is a separable reflexive Banach space with a uniformly convex dual $\mathbb{W}^*$, we can rely on the fact (8.12) in \cite{SPV2001} to assert that the norm $\|\cdot\|_{\mathbb{W}}$ is G$\hat{a}$teaux differentiable. Moreover, the G$\hat{a}$teaux derivative of the function $k(w)=\frac{1}{2}\|w\|_{\mathbb{W}}^2$ is the duality map, that is,
	$$\langle\partial_w k(w),z\rangle=\frac{1}{2}\frac{\mathrm{d}}{\mathrm{d}\varepsilon}\|w+\varepsilon z\|_{\mathbb{W}}^2\Big|_{\varepsilon=0}=\langle\mathscr{J}[w],z\rangle,$$ for $z\in\mathbb{W}$, where $\partial_w k(w)$ denotes the G\^ateaux derivative of $k$ at $w\in\mathbb{W}$.  Moreover, $\mathbb{U}$ is a separable Hilbert space whose dual space is same as $\mathbb{U}$, then the space $\mathbb{U}$ possesses a Fr\'echet differentiable norm (see Theorem 8.24 in \cite{SPV2001}). The below lemma provides the expression for the optimal control $u$.
	\begin{lem}\label{lem3.2}
		The optimal control ${u}$, satisfying \eqref{LEq1} and minimizing the cost functional \eqref{CF}, can be expressed as
		\begin{align*}
			{u}(t)=\mathrm{B}^{*}\mathscr{R}(T-t)^*\mathscr{J}\left[\mathrm{R}(\lambda,\Psi_0^{T})g(w(\cdot))\right],\ t\in J,
		\end{align*}
		where
		\begin{align*}
			g(w(\cdot))=\zeta_1-\mathscr{R}(T)\zeta.
		\end{align*}
	\end{lem}
	A proof of the lemma can be easily obtained by adapting the technique used in the proof of Lemma 3.4 of \cite{MCRF2021}. Subsequently, we establish the approximate controllability result of the linear control system \eqref{3.2} by using the above control.
	\begin{lem}\label{lem3.3}
		The linear control system \eqref{LEq1} is approximately controllable on $J$ if and only if Assumption \textit{(H0)} is satisfied.
	\end{lem}
	A proof of the above lemma is a straightforward adaptation of Theorem 3.2, \cite{NAHS2021}.
	
	\begin{rem}\label{rem3.4}
		If the operator $\Psi_{0}^{T}$ is positive, then Assumption (\textit{H0}) holds and vice versa (cf. Theorem 2.3, \cite{SIAM2003}). The positivity of $\Psi_{0}^{T}$ is equivalent to $$ \langle w^*, \Psi_{0}^{T}w^*\rangle=0\Rightarrow w^*=0.$$ Further, we have 
		\begin{align*}
			\langle w^*, \Psi_{0}^{T}w^*\rangle =\int_{0}^T\left\|\mathrm{B}^*\mathscr{R}(T-t)^*w^*\right\|_{\mathbb{U}}^2\mathrm{d}t.
		\end{align*}
		The above fact and Lemma \ref{lem3.3} ensure that the approximate controllability of the linear system \eqref{LEq1} is analogous to the condition $$\mathrm{B}^*\mathscr{R}(T-t)^*w^*=0,\ 0\le t\le T \Rightarrow w^*=0.$$ 
	\end{rem}
	\begin{rem}
		The equivalent conditions of the approximate controllability provided in the above remark also hold for general Banach space $\mathbb{W}$, that is, the linear system  \eqref{LEq1} is approximately controllable on $J$ if and only if 
		$$\mathrm{B}^*\mathscr{R}(T-t)^*w^*=0,\ 0\le t\le T \Rightarrow w^*=0, \ \mbox{for}\ w\in\mathbb{W}^*.$$
		This is followed by the fact that $Range(\mathrm{L}_T)^{\perp}=Ker(\mathrm{L}_T^*)$.
	\end{rem}
	\section{Semilinear Control problem}\label{SLCS}\setcounter{equation}{0} 
	The objective of this section is to study the approximate controllability of the neutral integro-differential equation given in \eqref{SEq}.
	To achieve this, we first establish the existence of a mild solution for the integro-differential system \eqref{SEq} associated with the control
	\begin{align}\label{Cot}
		u_{\lambda}(t)=u_{\lambda}(t;w)=\mathrm{B}^{*}\mathscr{R}(T-t)^*\mathscr{J}\left[\mathrm{R}(\lambda,\Psi_{0}^{T})l(w(\cdot))\right],\ t\in J,
	\end{align}
	where 
	\begin{align}\label{e4.2}
		l(w(\cdot))=\zeta_1-\mathscr{R}(T)\psi(0)-\int_{0}^{T}\mathscr{R}(t-s)\left[f'_1(s)+f_2(s)\right]\mathrm{d}s-\int_{0}^{T}\mathscr{R}(t-s)f(s,\tilde{w}_s)\mathrm{d}s,
	\end{align}
	with $\lambda>0$ and $\zeta_1\in\mathbb{W}$ and $\tilde{w}:(-\infty,T]\to\mathbb{W}$ such that $\tilde{w}(t)=\psi(t), t\in(-\infty,0)$ and $\tilde{w}(t)=w(t), t\in J$.
	\begin{theorem}\label{thm4.1}
		 Assume (H1)-(H2)  and $f_1\in\mathrm{W}^{1,1}(J;\mathbb{W}), f_2\in\mathrm{L}^1(J;\mathbb{W})$. Then for every $\lambda>0$ and any $\zeta_1\in\mathbb{W}$, there exists at least one mild solution for the system \eqref{SEq} under the control \eqref{Cot}.
	\end{theorem}
	\begin{proof}
		Consider a set $Z_{\psi}:=\{w\in C(J;\mathbb{W}) : w(0)=\psi(0)\}$ equipped with the norm  $\left\|\cdot\right\|_{C(J;\mathbb{W})}$. For each $r>0$, let us define $W_r=\{w\in Z_{\psi}: \left\|w\right\|_{C(J;\mathbb{W})}\le r\}$.
		Next, for any $\lambda>0$, define an operator $\mathcal{T}_{\lambda}:Z_{\psi}\to Z_{\psi}$ by $(\mathcal{T}_{\lambda}w)(t)=z(t)$, where 
		\begin{align*}
			z(t)=\mathscr{R}(t)\psi(0)+\int_{0}^{t}\mathscr{R}(t-s)\left[\mathrm{B}u_{\lambda}(s)+f(s,\tilde{w}_s)\right]\mathrm{d}s+\int_{0}^{t}\mathscr{R}(t-s)\left[f'_1(s)+f_2(s)\right]\mathrm{d}s, t\in J.
		\end{align*}
		Here $u_{\lambda}$ is given as in \eqref{Cot}. From the definition of $\mathcal{T}_{\lambda}$, it is evident that the system \eqref{SEq} possesses a mild solution, if the operator $ \mathcal{T}_{\lambda}$ admits a fixed point. The proof of the existence of a fixed point for the operator $\mathcal{T}_{\lambda}$ is divided into the following steps. 
		\vskip 0.1in 
		\noindent\textbf{Step 1: } \emph{For an arbitrary $\lambda>0$, there is a $r=r(\lambda)>0$ such that $ \mathcal{T}_{\lambda}(W_r)\subset W_r$}. Let us first calculate
		\begin{align*}
			\|u_{\lambda}(t)\|_{\mathbb{U}}&=\left\|\mathrm{B}^{*}\mathscr{R}(T-t)^*\mathscr{J}\left[\mathrm{R}(\lambda,\Psi_{0}^{T})l(w(\cdot))\right]\right\|_{\mathbb{U}}\nonumber\\&\le\frac{KM}{\lambda}\|l(w(\cdot))\|_{\mathbb{W}}\nonumber\\&\le \frac{KM}{\lambda}\left[\|\zeta_1\|_{\mathbb{W}}+K\|\psi(0)\|_{\mathbb{W}}+K\int_{0}^{T}\|f'_1(s)+f_2(s)\|_{\mathbb{W}}\mathrm{d}s+K\int_{0}^{T}\|f(s,\tilde{w}_s)\|_{\mathbb{W}}\mathrm{d}s\right]\nonumber\\&\le \frac{KM}{\lambda}\left[\|\zeta_1\|_{\mathbb{W}}+K\|\psi(0)\|_{\mathbb{W}}+K\|f'_1(s)+f_2(s)\|_{\mathrm{L}^1(J;\mathbb{W})}\mathrm{d}s+K\int_{0}^{T}\gamma(s)\mathrm{d}s\right]\nonumber\\&\le \frac{KM\tilde{M}}{\lambda},
		\end{align*}
		where $\tilde{M}=\|\zeta_1\|_{\mathbb{W}}+K\|\psi(0)\|_{\mathbb{W}}+K\|\gamma\|_{\mathrm{L}^1(J;\mathbb{R}^+)}+K\|f'_1+f_2\|_{\mathrm{L}^1(J;\mathbb{W})}$. The above inequality implies that $\|u_{\lambda}(t)\|_{\mathbb{U}}$ is bounded for all $t\in J$. We now compute
		\begin{align}\label{4.3}
			&\left\|(\mathcal{T}_{\lambda}w)(t)\right\|_\mathbb{W}\nonumber\\&=\left\|\mathscr{R}(t)\psi(0)+\int_{0}^{t}\mathscr{R}(t-s)\left[\mathrm{B}u_{\lambda}(s)+f(s,\tilde{w}_s)\right]\mathrm{d}s+\int_{0}^{t}\mathscr{R}(t-s)\left[f'_1(s)+f_2(s)\right]\mathrm{d}s\right\|_{\mathbb{W}}\nonumber\\&\le K\|\psi(0)\|_\mathbb{W}+\frac{K^2M^2\tilde{M}}{\lambda}+ K\|\gamma\|_{\mathrm{L}^1(J;\mathbb{R}^+)}+K\|f'_1+f_2\|_{\mathrm{L}^1(J;\mathbb{W})}.
		\end{align}
	From the inequality \eqref{4.3}, we see that, for each $\lambda>0$, there is a large $r=r(\lambda)>0$ so that  $ \mathcal{T}_{\lambda}(W_r)\subset W_r$.
		\vskip 0.1in 
		\noindent\textbf{Step 2: } \emph{The operator $ \mathcal{T}_{\lambda}$ is continuous}. To accomplish this goal, we consider a sequence $\{{w}^n\}^\infty_{n=1}\subseteq W_r$ such that ${w}^n\rightarrow {w}\ \mbox{in}\ W_r,$ that is,
		$$\lim\limits_{n\rightarrow \infty}\left\|w^n-w\right\|_{C(J;\mathbb{W})}=0.$$
		Using the axiom (A1), we get
		\begin{align*}
			\left\|\tilde{w_{t}^n}-\tilde{w_{t}}\right\|_{\mathfrak{B}}\le\varLambda(t)\sup\limits_{0\le s\le t}\left\|w^n(s)-w(s)\right\|_{\mathbb{W}}\nonumber\le H_1\left\|w^n-w\right\|_{C(J;\mathbb{W})}\to 0 
		\end{align*}
as $ n\to\infty$ for $ t\in J,$	where $\sup_{0\le t\le T }|\varLambda(t)|\le H_1$.	Using the above convergence, Assumption \ref{ass2.4} (\textit{H2}) and Lebesgue's dominant convergence theorem, we deduce that
		\begin{align*}
			\left\|l(w^{n}(\cdot))-l(w(\cdot))\right\|_{\mathbb{W}}&\le \left\|\int^{T}_{0}\mathscr{R}(T-s)\left[f(s, \tilde{w_{s}^n})-f(s,\tilde{w}_{s})\right]\mathrm{d}s\right\|_{\mathbb{W}}\nonumber\\&\le K\int^{T}_{0}\left\|f(s, \tilde{w_{s}^n})-f(s,\tilde{w}_{s})\right\|_{\mathbb{W}}\mathrm{d}s\to 0 \ \mbox{as}\ n\to\infty.
		\end{align*}
		From Lemma \ref{lem2.17}, it follows that the mapping $\mathrm{R}(\lambda,\Psi_{0}^{T}):\mathbb{W}\to\mathbb{W}$  is uniformly continuous on every bounded subset of $\mathbb{W}$. Thus, we have
		$$\mathrm{R}(\lambda,\Psi_{0}^{T})l(w^{n}(\cdot))\to\mathrm{R}(\lambda,\Psi_{0}^{T})l(w(\cdot))\ \mbox{in}\  \mathbb{W} \ \mbox{as}\ n\to\infty.$$ As the mapping $\mathscr{J}:\mathbb{W}\to\mathbb{W}^{*}$  is demicontinuous, it follows immediately that 
		\begin{align*}
			\mathscr{J}\left[\mathrm{R}(\lambda,\Psi_{0}^{T})l(w^{n}(\cdot))\right]\xrightharpoonup{w}\mathscr{J}\left[\mathrm{R}(\lambda,\Psi_{0}^{T})l(w(\cdot))\right]\ \text{ in }\ \mathbb{W}^{*} \ \text{as} \ n\to\infty.
		\end{align*}
		Given Assumption (\textit{H1}) and Lemma \ref{lem2.6}, it follows that the operator $\mathscr{R}(t)$ is compact for each $t>0$. Consequently, the operator $\mathscr{R}(t)^*$ is also compact for each $t>0$. Therefore, using the aforementioned weak convergence alongside the compactness of the operator $\mathscr{R}(t)^*$, one can readily arrive at
		\begin{align}\label{e4.4}
			\left\|\mathscr{R}(T-t)^*\mathscr{J}\left[\mathrm{R}(\lambda,\Psi_{0}^{T})l(w^n(\cdot))\right]\!\!-\!\mathscr{R}(T-t)^*\mathscr{J}\left[\mathrm{R}(\lambda,\Psi_{0}^{T})l(w(\cdot))\right]\right\|_{\mathbb{W^*}}\!\!\!\!\!\to 0 
		\end{align}
as $n\to\infty$	for all $t\in [0,T)$. Using \eqref{Cot} and \eqref{e4.4}, we easily get 
		\begin{align}\label{4.4.9}
			&\left\|u^{n}_{\lambda}(t)-u_{\lambda}(t)\right\|_{\mathbb{U}}\nonumber\\&=\left\|\mathrm{B}^{*}\mathscr{R}(T-t)^*\mathscr{J}\left[\mathrm{R}(\lambda,\Psi_{0}^{T})l(w^n(\cdot))\right]-\mathrm{B}^{*}\mathscr{R}(T-t)^*\mathscr{J}\left[\mathrm{R}(\lambda,\Psi_{0}^{T})l(w(\cdot))\right]\right\|_{\mathbb{U}}\nonumber\\&\leq \left\|\mathrm{B}^*\right\|_{\mathcal{L}(\mathbb{W}^{*};\mathbb{U})}\left\|\mathscr{R}(T-t)^*\mathscr{J}\left[\mathrm{R}(\lambda,\Psi_{0}^{T})l(w^n(\cdot))\right]-\mathscr{R}(T-t)^*\mathscr{J}\left[\mathrm{R}(\lambda,\Psi_{0}^{T})l(w(\cdot))\right]\right\|_{\mathbb{W^*}}\nonumber\\&\leq M\left\|\mathscr{R}(T-t)^*\mathscr{J}\left[\mathrm{R}(\lambda,\Psi_{0}^{T})l(w^n(\cdot))\right]-\mathscr{R}(T-t)^*\mathscr{J}\left[\mathrm{R}(\lambda,\Psi_{0}^{T})l(w(\cdot))\right]\right\|_{\mathbb{W^*}}\nonumber\\&\to 0 \ \text{ as } \ n\to\infty, \text{ uniforlmy for all }\ t\in [0,T).
		\end{align}
		Using  \eqref{4.4.9}, Assumption \ref{ass2.4} (\textit{H2}) and Lebesgue's dominate convergence theorem, we obtain 
		\begin{align*}
			&\left\|(\mathcal{T}_{\lambda}w^{n})(t)-(\mathcal{T}_{\lambda}w)(t)\right\|_{\mathbb{W}}\nonumber\\&\leq\left\|\int_{0}^{t}\mathscr{R}(t-s)\mathrm{B}\left[u^n_{\lambda}(s)-u_{\lambda}(s)\right]\mathrm{d}s\right\|_{\mathbb{W}}+\left\|\int_{0}^{t}\mathscr{R}(t-s)\left[f(s, \tilde{w_{s}^n})-f(s,\tilde{w}_{s})\right]\mathrm{d}s\right\|_{\mathbb{W}}\nonumber\\&\leq MKT\esssup_{t\in J}\left\|u^n_{\lambda}(t)-u_{\lambda}(t)\right\|_{\mathbb{U}} +K\int_{0}^{t}\left\|\left[f(s,\tilde{w}_{s}^{n})-f(s,\tilde{w}_{s})\right]\right\|_{\mathbb{W}}\mathrm{d}s\nonumber\\&\to 0
		\end{align*}
as $ n\to\infty$	for each $t\in J$. Hence, the map $\mathcal{T}_{\lambda}$ is continuous.
		\vskip 0.1in 
		\noindent\textbf{Step 3: } In this step, we will show that  \emph{$ \mathcal{T}_{\lambda}$ is a compact operator} for $\lambda>0$. To validate this claim, we initially prove that the image of $W_r$ under $\mathcal{T}_{\lambda}$ is equicontinuous. To proceed this, let $0\le \tau_1\le \tau_2\le T$ and any $w\in W_r$, we compute
		\begin{align*}
			&\left\|(\mathcal{T}_{\lambda}w)(\tau_2)-(\mathcal{T}_{\lambda}w)(\tau_1)\right\|_{\mathbb{W}}\nonumber\\&\leq\left\|\left[\mathscr{R}(\tau_2)-\mathscr{R}(\tau_1)\right]\psi(0)\right\|_{\mathbb{W}}+\left\|\int_{\tau_{1}}^{\tau_{2}}\mathscr{R}(\tau_2-s)\left[\mathrm{B}u_{\lambda}(s)+f(s,\tilde{w}_{s})\right]\mathrm{d}s\right\|_{\mathbb{W}}\nonumber\\&\quad+\left\|\int_{\tau_{1}}^{\tau_{2}}\mathscr{R}(\tau_2-s)\left[f'_1(s)+f_2(s)\right]\mathrm{d}s\right\|_{\mathbb{W}}+\left\|\int_{0}^{\tau_{1}}\left[\mathscr{R}(\tau_2-s)-\mathscr{R}(\tau_1-s)\right]\left[f'_1(s)+f_2(s)\right]\mathrm{d}s\right\|_{\mathbb{W}}\nonumber\\&\quad+\left\|\int_{0}^{\tau_{1}}\left[\mathscr{R}(\tau_2-s)-\mathscr{R}(\tau_1-s)\right]\left[\mathrm{B}u_{\lambda}(s)+f(s,\tilde{w}_{s})\right]\mathrm{d}s\right\|_{\mathbb{W}}\nonumber\\&\leq\left\|\left[\mathscr{R}(\tau_2)-\mathscr{R}(\tau_1)\right]\psi(0)\right\|_{\mathbb{W}}+\frac{K^2M^2\tilde{M}}{\lambda}(\tau_2-\tau_1)+K\int_{\tau_1}^{\tau_2}\!\!\!\gamma(s)\mathrm{d}s+K\int_{\tau_{1}}^{\tau_{2}}\!\!\!\left\|f'_1(s)+f_2(s)\right\|_{\mathbb{W}}\mathrm{d}s\nonumber\\&\quad+\int_{0}^{\tau_{1}}\left\|\left[\mathscr{R}(\tau_2-s)-\mathscr{R}(\tau_1-s)\right]\left[f'_1(s)+f_2(s)+\mathrm{B}u_{\lambda}(s)+f(s,\tilde{w}_{s})\right]\right\|_{\mathbb{W}}\mathrm{d}s
		\end{align*}
		If $\tau_1=0$, then by the above estimate, we obtain that 
		\begin{align*}
			\lim_{\tau_2\to 0^+}\left\|(\mathcal{T}_{\lambda}w)(\tau_2)-(\mathcal{T}_{\lambda}w)(\tau_1)\right\|_{\mathbb{W}}=0,\ \mbox{ unifromly for}\  w\in W_r.
		\end{align*}
		For given $\epsilon>0$, let us take  $\epsilon<\tau_1<T$, we have
		\begin{align}\label{e4.8}
			&\left\|(\mathcal{T}_{\lambda}w)(\tau_2)-(\mathcal{T}_{\lambda}w)(\tau_1)\right\|_{\mathbb{W}}\nonumber\\&\leq\left\|\left[\mathscr{R}(\tau_2)-\mathscr{R}(\tau_1)\right]\psi(0)\right\|_{\mathbb{W}}+\frac{K^2M^2\tilde{M}}{\lambda}(\tau_2-\tau_1)+K\int_{\tau_1}^{\tau_2}\!\!\!\gamma(s)\mathrm{d}s+K\int_{\tau_{1}}^{\tau_{2}}\!\!\!\left\|f'_1(s)+f_2(s)\right\|_{\mathbb{W}}\mathrm{d}s\nonumber\\&\quad+\int_{0}^{\tau_{1-\epsilon}}\left\|\left[\mathscr{R}(\tau_2-s)-\mathscr{R}(\tau_1-s)\right]\left[f'_1(s)+f_2(s)+\mathrm{B}u_{\lambda}(s)+f(s,\tilde{w}_{s})\right]\right\|_{\mathbb{W}}\mathrm{d}s\nonumber\\&\quad+\int_{\tau_1-\epsilon}^{\tau_{1}}\left\|\left[\mathscr{R}(\tau_2-s)-\mathscr{R}(\tau_1-s)\right]\left[f'_1(s)+f_2(s)+\mathrm{B}u_{\lambda}(s)+f(s,\tilde{w}_{s})\right]\right\|_{\mathbb{W}}\mathrm{d}s\nonumber\\&\leq\left\|\left[\mathscr{R}(\tau_2)-\mathscr{R}(\tau_1)\right]\psi(0)\right\|_{\mathbb{W}}+\frac{K^2M^2\tilde{M}}{\lambda}(\tau_2-\tau_1)+K\int_{\tau_1}^{\tau_2}\!\!\!\gamma(s)\mathrm{d}s+K\int_{\tau_{1}}^{\tau_{2}}\!\!\!\left\|f'_1(s)+f_2(s)\right\|_{\mathbb{W}}\mathrm{d}s\nonumber\\&\quad+\sup_{s\in[0,\tau_1-\epsilon]}\left\|\mathscr{R}(\tau_2-s)-\mathscr{R}(\tau_1-s)\right\|_{\mathcal{L}(\mathbb{W})}\int_{0}^{\tau_{1}-\epsilon}\left\|\left[f'_1(s)+f_2(s)+\mathrm{B}u_{\lambda}(s)+f(s,\tilde{w}_{s})\right]\right\|_{\mathbb{W}}\mathrm{d}s\nonumber\\&\quad+2K\int_{\tau_1-\epsilon}^{\tau_{1}}\left[\left\|f'_1(s)+f_2(s)\right\|_{\mathbb{W}}+\gamma(s)\right]\mathrm{d}s+\frac{2K^2M^2\tilde{M}}{\lambda}\epsilon
		\end{align} 
		Analogous to the estimate \eqref{e2.20}, it can be readily observed that the right-hand side of expressions \eqref{e4.8} converges to zero as $|\tau_2-\tau_1| \rightarrow 0$ and arbitrariness of $\epsilon$.  Thus, the image of $W_r$ under $\mathcal{T}_{\lambda}$ is equicontinuous.
		
		Furthermore, we claim that for each $t\in J$, the set $\mathfrak{X}(t)=\{(\mathcal{T}_\lambda w)(t):w\in W_r\}$ is relatively compact. For $t=0$, the claim is straightforward. Now, consider a fixed  $ 0<t\leq T$ and let $\eta$ be given with $ 0<\eta<t$, we define
		\begin{align*}
			(\mathcal{T}_{\lambda}^{\eta}w)(t)=&\mathscr{R}(\eta)\Big[\mathscr{R}(t-\eta)\psi(0)+\int_{0}^{t-\eta}\mathscr{R}(t-s-\eta)\left[\mathrm{B}u_{\lambda}(s)+f(s,\tilde{w}_s)\right]\mathrm{d}s\nonumber\\&\quad\qquad+\int_{0}^{t-\eta}\mathscr{R}(t-s-\eta)\left[f'_1(s)+f_2(s)\right]\mathrm{d}s\Big]=\mathscr{R}(\eta)y(t-\eta).
		\end{align*}
		Thus, the set  $\mathfrak{X}_{\eta}(t)=\{(\mathcal{T}^{\eta}_\lambda w)(t):w\in W_r\}$ is relatively compact in $\mathbb{W}$ follows by the  compactness of the operator $\mathscr{R}(\eta)$. Consequently, there exist a finite $ z_{i}$'s, for $i=1,\dots, n $ in $ \mathbb{W} $ such that 
		\begin{align*}
			\mathfrak{X}_{\eta}(t) \subset \bigcup_{i=1}^{n}\mathcal{S}_{z_i}(\varepsilon/2),
		\end{align*}
		for some $\varepsilon>0$. We choose $\eta>0$ such that 
		\begin{align*}
			&\left\|(\mathcal{T}_{\lambda}w)(t)-(\mathcal{T}_{\lambda}^{\eta}w)(t)\right\|_{\mathbb{W}}\nonumber\\&\le\left\|\left[\mathscr{R}(t)-\mathscr{R}(\eta)\mathscr{R}(t-\eta)\right]\psi(0)\right\|_{\mathbb{W}}+\int_{t-\eta}^{t}\|\mathscr{R}(t-s)\left[\mathrm{B}u_{\lambda}(s)+f(s,\tilde{w}_s)+f'_1(s)+f_2(s)\right]\|_{\mathbb{W}}\mathrm{d}s\nonumber\\&\quad+\int_{0}^{t-\eta}\|\left[\mathscr{R}(t-s)-\mathscr{R}(\eta)\mathscr{R}(t-s-\eta)\right]\left[\mathrm{B}u_{\lambda}(s)+f(s,\tilde{w}_s)+f'_1(s)+f_2(s)\right]\|_{\mathbb{W}}\mathrm{d}s\nonumber\\&\le\left\|\left[\mathscr{R}(t)-\mathscr{R}(\eta)\mathscr{R}(t-\eta)\right]\psi(0)\right\|_{\mathbb{W}}+\frac{K^2M^2\tilde{M}}{\lambda}\eta+K\int_{t-\eta}^{t}\left[\gamma(s)+f'_1(s)+f_2(s)\right]\mathrm{d}s\\&\quad+\int_{0}^{t-2\eta}\|\left[\mathscr{R}(t-s)-\mathscr{R}(\eta)\mathscr{R}(t-s-\eta)\right]\left[\mathrm{B}u_{\lambda}(s)+f(s,\tilde{w}_s)+f'_1(s)+f_2(s)\right]\|_{\mathbb{W}}\mathrm{d}s\\&\quad+\int_{t-2\eta}^{t-\eta}\|\left[\mathscr{R}(t-s)-\mathscr{R}(\eta)\mathscr{R}(t-s-\eta)\right]\left[\mathrm{B}u_{\lambda}(s)+f(s,\tilde{w}_s)+f'_1(s)+f_2(s)\right]\|_{\mathbb{W}}\mathrm{d}s\nonumber\\&\le\left\|\left[\mathscr{R}(t)-\mathscr{R}(\eta)\mathscr{R}(t-\eta)\right]\psi(0)\right\|_{\mathbb{W}}+\frac{K^2M^2\tilde{M}}{\lambda}\eta+K\int_{t-\eta}^{t}\left[\gamma(s)+f'_1(s)+f_2(s)\right]\mathrm{d}s\\&\quad+\int_{0}^{t-2\eta}\left\|\mathscr{R}(t-s)-\mathscr{R}(\eta)\mathscr{R}(t-s-\eta)\right\|_{\mathcal{L}(\mathbb{W})}\left\|\mathrm{B}u_{\lambda}(s)+f(s,\tilde{w}_s)+f'_1(s)+f_2(s)\right\|_{\mathbb{W}}\mathrm{d}s\\&\quad+(K+K^2)\int_{t-2\eta}^{t-\eta}\left\|\mathrm{B}u_{\lambda}(s)+f(s,\tilde{w}_s)+f'_1(s)+f_2(s)\right\|_{\mathbb{W}}\mathrm{d}s\le \frac{\varepsilon}{2}.
		\end{align*}
		Therefore $$\mathfrak{X}(t)\subset \bigcup_{i=1}^{n}\mathcal{S}_{z_i}(\varepsilon ).$$
		Thus the set $\mathfrak{X}(t)$ is relatively compact in $ \mathbb{W}$ for each $t\in [0,T]$. Hence, by employing  Arzel\'a-Ascoli theorem, we develop the compactness of $ \mathcal{T}_{\lambda}$. 
		
		Furthermore,  utilizing the \emph{Schauder fixed point theorem}, we infer that for each $\lambda>0$, the operator $\mathcal{T}_{\lambda}$  possesses a fixed point in $W_{r}$.
	\end{proof}
	We now proceed to establish the approximate controllability result for the system \eqref{SEq} in the following theorem.
	\begin{theorem}\label{thm4.4}
		Suppose that Assumptions (H0)-(H2) are satisfied. Then the system \eqref{SEq} is approximately controllable.
	\end{theorem}
	\begin{proof}
		Theorem \ref{thm4.1} implies the  existence of a mild solution, say,  $w^{\lambda}\in W_{r(\lambda)}$ for each $\lambda>0$, and any $\zeta_1\in \mathbb{W}$. It is immediate that $w^{\lambda}_0=\psi$ and 
		\begin{align*}
			w^{\lambda}(t)=\mathscr{R}(t)\psi(0)+\int_{0}^{t}\mathscr{R}(t-s)\left[\mathrm{B}u_{\lambda}(s)+f(s,w^\lambda_s)\right]\mathrm{d}s+\int_{0}^{t}\mathscr{R}(t-s)\left[f'_1(s)+f_2(s)\right]\mathrm{d}s,
		\end{align*}
		for $t\in J$, with the control 
		\begin{align}
			u_{\lambda}(t)=\mathrm{B}^{*}\mathscr{R}(T-t)^*\mathscr{J}\left[\mathrm{R}(\lambda,\Psi_{0}^{T})l(w^{\lambda}(\cdot))\right],\ t\in J,
		\end{align}
		where 
		\begin{align*}
			l(w^{\lambda}(\cdot))=\zeta_1-\mathscr{R}(t)\psi(0)-\int_{0}^{T}\mathscr{R}(t-s)f(s,w^{\lambda}_s)\mathrm{d}s-\int_{0}^{T}\mathscr{R}(t-s)\left[f'_1(s)+f_2(s)\right]\mathrm{d}s.
		\end{align*}
		Next, we evaluate
		\begin{align}\label{4.35}
			w^{\lambda}(T)&=\mathscr{R}(T)\psi(0)+\int_{0}^{T}\mathscr{R}(T-s)\left[\mathrm{B}u_{\lambda}(s)+f(s,w^{\lambda}_s)\right]\mathrm{d}s+\int_{0}^{T}\mathscr{R}(T-s)\left[f'_1(s)+f_2(s)\right]\mathrm{d}s\nonumber\\&=\mathscr{R}(T)\psi(0)+\int_{0}^{T}\mathscr{R}(T-s)f(s,w^{\lambda}_s)\mathrm{d}s+\int_{0}^{T}\mathscr{R}(T-s)\left[f'_1(s)+f_2(s)\right]\mathrm{d}s\nonumber\\&\quad+\int_{0}^{T}\mathscr{R}(T-s)\mathrm{B}\mathrm{B}^*\mathscr{R}(T-s)^*\mathscr{J}\left[\mathrm{R}(\lambda,\Psi_{0}^{T})l(w^\lambda(\cdot))\right]\mathrm{d}s\nonumber\\&=\mathscr{R}(T)\psi(0)+\int_{0}^{T}\mathscr{R}(T-s)f(s,w^{\lambda}_s)\mathrm{d}s+\int_{0}^{T}\mathscr{R}(T-s)\left[f'_1(s)+f_2(s)\right]\mathrm{d}s\nonumber\\&\quad+\Psi_{0}^{T}\mathscr{J}\left[\mathrm{R}(\lambda,\Psi_{0}^{T})l(w^\lambda(\cdot))\right]\nonumber\\&=\zeta_1-\lambda\mathrm{R}(\lambda,\Psi_{0}^{T})l(w^\lambda(\cdot)).
		\end{align}	
		Using Assumption  \textit{(H2)}, we get
		\begin{align}
			\int_{0}^{T}\left\|f(s,w^{\lambda_i}_s)\right\|_{\mathbb{W}}\mathrm{d}s&\le \int_{0}^{T}\gamma(s)\mathrm{d} s<+\infty,i\in\mathbb{N}, \nonumber
		\end{align}
		The above fact ensures that the sequence $\{f(\cdot, w^{\lambda_{i}}_{(\cdot)})\}_{i=1}^{\infty}$ is uniformly integrable. Then by the application of Dunford-Pettis theorem, we can find a subsequence of $ \{f(\cdot, w^{\lambda_{i}}_{(\cdot)})\}_{i=1}^{\infty}$ still denoted by $ \{f(\cdot, w^{\lambda_{i}}_{(\cdot)})\}_{i=1}^{\infty}$ such that
		\begin{align*}
			f(\cdot,w^{\lambda_{i}}_{(\cdot)})\xrightharpoonup{w}f(\cdot) \ \mbox{in}\ \mathrm{L}^1(J;\mathbb{W}), \ \mbox{as}\  \lambda_i\to 0^+ \ (i\to\infty).
		\end{align*}
		Next, we see that
		\begin{align}\label{e4}
			\left\|l(w^{\lambda_{i}}(\cdot))-\xi\right\|_{\mathbb{W}}&\le\left\|\int_{0}^{T}\mathscr{R}(T-s)\left[f(s,w^{\lambda_i}_s)-f(s)\right]\mathrm{d}s\right\|_{\mathbb{W}}\to 0
		\end{align}
	as $\lambda_i\to0^+ \ (i\to\infty),$ where 
		\begin{align*}
			\xi &=\zeta_1-\mathscr{R}(T)\psi(0)-\int_{0}^{T}\mathscr{R}(T-s)f(s)\mathrm{d}s-\int_{0}^{T}\mathscr{R}(T-s)\left[f'_1(s)+f_2(s)\right]\mathrm{d}s.
		\end{align*}
		The estimate \eqref{e4} goes to zero using the above weak convergences together with Corollary \ref{cor1}. 	The equality \eqref{4.35} guarantees that
		$z_{\lambda_i}=w^{\lambda_i}(T)-\zeta_1$ for each $\lambda_i>0,\ i\in\mathbb{N}$, is a solution of the equation $$\lambda_i z_{\lambda_i}+\Psi_{0}^{T}\mathscr{J}[z_{\lambda_i}]=\lambda_ih_{\lambda_i},$$ where 
		\begin{align*}
			h_{\lambda_i}=-		l(w^{\lambda_i}(\cdot))&=\mathscr{R}(T)\psi(0)+\int^{T}_{0}\mathscr{R}(T-s)f(s,w^{\lambda_i}_s)\mathrm{d}s+\mathscr{R}(T-s)\left[f'_1(s)+f_2(s)\right]\mathrm{d}s-\zeta_1.
		\end{align*}
		Note that Assumption \textit{(H0)} ensures the the operator $\Psi_0^T$ is positive. By applying Theorem 2.5 from \cite{SIAM2003} along with the estimate \eqref{e4}, we get
		\begin{align*}
			\nonumber\left\|w^{\lambda_i}(T)-\zeta_1\right\|_{\mathbb{W}}=\left\|z_{\lambda_i}\right\|_{\mathbb{W}}\to 0\ \mbox{as}\  \lambda_i \rightarrow 0^{+}\ (i\to\infty).
		\end{align*}
		Hence, the system \eqref{SEq} is approximately controllable.
	\end{proof}
	We now prove the approximate controllability of the system \eqref{SEq} within the framework of a general Banach space $\mathbb{W}$. To demonstrate this, we impose the following assumption on the nonlinear term $f(\cdot,\cdot)$.
	\begin{enumerate}
		\item [\textit{(H3)}] The mapping $f:J\times\mathfrak{B}\to\mathbb{W}$ is continuous and there exists a function $\beta\in\mathrm{L}^{1}(J;\mathbb{R}^+)$ such that
		$$\|f(t,\phi_1)-f(t,\phi_2)\|_{\mathbb{W}}\le\beta(t)\|\phi_1-\phi_2\|_{\mathfrak{B}}, \ t\in J, \phi_1,\phi_2\in\mathfrak{B}.$$
	\end{enumerate}
	\begin{theorem}\label{thm3.6}
		If Assumption \textit{(H3)} holds, then for any $u\in\mathrm{L}^{2}(J;\mathbb{U})$, the control system \eqref{SEq} admit a unique mild solution, provided $$KH_1\|\beta\|_{\mathrm{L}^1(J;\mathbb{R}^+)}<1.$$
	\end{theorem}
	\begin{proof}
		We consider a set $Z_{\psi}:=\{w\in C(J;\mathbb{W}) : w(0)=\psi(0)\}$ endowed the norm $\left\|\cdot\right\|_{C(J;\mathbb{W})}$.	We now define an operator $\mathcal{T}:Z_{\psi}\to Z_{\psi}$ as
		$(\mathcal{T}w)(t)=z(t)$, where 
		\begin{align*}
			z(t)=\mathscr{R}(t)\psi(0)+\int_{0}^{t}\mathscr{R}(t-s)\left[\mathrm{B}u(s)+f(s,\tilde{w}_s)\right]\mathrm{d}s+\int_{0}^{t}\mathscr{R}(t-s)\left[f'_1(s)+f_2(s)\right]\mathrm{d}s, t\in J.
		\end{align*}
	Here $\tilde{w}:(-\infty,T]\to\mathbb{W}$ be such that $\tilde{w}(t)=\psi(t), t\in(-\infty,0)$ and $\tilde{w}(t)=w(t), t\in J$.
		It is clear that the system \eqref{SEq} has a mild solution when the operator $\mathcal{T}$ has a fixed point. By the continuity of the function $f(\cdot,\cdot)$, one can easily see that $\mathcal{T}(Z_{\psi})\subset Z_{\psi}$. 	Further, using the axiom (A1), we estimate
		\begin{align*}
			\left\|\tilde{w_{t}}-\tilde{x_{t}}\right\|_{\mathfrak{B}}\le\varLambda(t)\sup\limits_{0\le s\le t}\left\|w(s)-x(s)\right\|_{\mathbb{W}}\nonumber\le H_1\left\|w-x\right\|_{C(J;\mathbb{W})}, 
		\end{align*}
for any $ w,x\in Z_{\psi}, $	where $\sup_{0\le t\le T }|\varLambda(t)|\le H_1$. Using the above inequality together with Assumption \textit{(H3)}, we obtain
		\begin{align*}
			\left\|(\mathcal{T}w)(t)-(\mathcal{T}x)(t)\right\|_{\mathbb{W}}&\leq\left\|\int_{0}^{t}\mathscr{R}(t-s)\left[f(s, \tilde{w_{s}})-f(s,\tilde{x}_{s})\right]\mathrm{d}s\right\|_{\mathbb{W}}\nonumber\\&\le K\int_{0}^{t}\left\|f(s,\tilde{w}_{s})-f(s,\tilde{x}_{s})\right\|_{\mathbb{W}}\mathrm{d}s\nonumber\\&\le K\int_{0}^{t}\beta(s)\left\|\tilde{w_{s}}-\tilde{x_{s}}\right\|_{\mathfrak{B}}\mathrm{d}s\nonumber\\&\le KH_1\|\beta\|_{\mathrm{L}^1(J;\mathbb{R}^+)}\|w-x\|_{C(J;\mathbb{W})}.
		\end{align*}
		The preceding fact implies that $\mathcal{T}$ is a contraction map on $Z_\psi$. Therefore, by invoking the \emph{Banach fixed point theorem} we deduce the existence of a unique fixed point of the operator $\mathcal{T}$ which is a mild solution of the equation \eqref{SEq}.
	\end{proof}
	\begin{theorem}
		If Assumption \textit{(H3)} is satisfied and the linear system \eqref{LEq1} is approximately controllable on $[0,t]$ for any $0<t\le T$, then the system \eqref{SEq} is approximately controllable.
	\end{theorem}
	\begin{proof}
		It follows from Theorem \ref{thm3.6} that the system \eqref{SEq} has a unique mild solution on $J$. For a fixed $\psi\in\mathfrak{B}$, let $z(\cdot)=z(\cdot,\psi,0)$ be a mild solution of equation \eqref{SEq} corresponding to the control $u=0$. It is immediate  to write
		\begin{equation*}
			z(t)=\begin{dcases}
				\psi(t), t\in(-\infty, 0),\\
				\mathscr{R}(t)\psi(0)+\int_{0}^{t}\mathscr{R}(t-s)f(s,z_s)\mathrm{d}s+\int_{0}^{t}\mathscr{R}(t-s)\left[f'_1(s)+f_2(s)\right]\mathrm{d}s, t\in J.
			\end{dcases}
		\end{equation*}
		Let us take a sequence $0<\tau_n<T$ such that $\tau_n\to T \ \mbox{as}\ n\to \infty$ and we denote by $z^n=z(\tau_n,\psi,0)$. Next, we consider the following equation:
		\begin{equation}\label{LEq11}
			\left\{
			\begin{aligned}
				\frac{\mathrm{d}}{\mathrm{d}t}\left[w(t)+\int_{0}^{t}\mathrm{G}(t-s)w(s)\mathrm{d}s\right]&=\mathrm{A}w(t)+\int_{0}^{t}\mathrm{N}(t-s)w(s)\mathrm{d}s\\&\qquad+\mathrm{B}u(t),\ t\in(0,T-\tau_n],\\
				w(0)&=z^n.
			\end{aligned}
			\right.
		\end{equation}
		Since the linear system \eqref{LEq11} is approximately controllable on $[0,T-\tau_n]$. Consequently, for any $\zeta_1\in\mathbb{W}$, there is a control function $v^n\in\mathrm{L}^2([0,T-\tau_n];\mathbb{U})$ as follow 
		\begin{align}\label{eq5}
			\lim_{n\to\infty}\left\|\mathscr{R}(T-\tau_n)z^n+\int_{0}^{T-\tau_n}\mathscr{R}(T-\tau_n-s)\mathrm{B}v^n(s)\mathrm{d}s-\zeta_1\right\|_{\mathbb{W}}=0.
		\end{align}
		Let 
		\begin{equation*}
			u^n(t)=\begin{dcases}
				0, \ \ 0\le t\le \tau_n,\\
				v^n(t-\tau_n),\  \tau_n<t\le T.
			\end{dcases}
		\end{equation*}
		For each $u^{n}(\cdot)$, there exists a unique solution $w^n:(-\infty,T)\to \mathbb{W}$ of the integral equation
		\begin{equation}\label{MS}
			w^n(t)=\begin{dcases}
				\psi(t),\  t\in(-\infty, 0),\\
				\mathscr{R}(t)\psi(0)+\int_{0}^{t}\mathscr{R}(t-s)\left[\mathrm{B}u^n(s)+f(s,w^n_s)\right]\mathrm{d}s\\\quad+\int_{0}^{t}\mathscr{R}(t-s)\left[f'_1(s)+f_2(s)\right]\mathrm{d}s,\ t\in J.
			\end{dcases}
		\end{equation}
		Since the solution of above equation is unique, we get $z(s)=w^n(s)$ for all $-\infty< s\le \tau_n$. From equation \eqref{MS}, we have
		\begin{align*}
			w^n(T)&=	\mathscr{R}(T)\psi(0)+\int_{0}^{T}\mathscr{R}(T-s)\left[\mathrm{B}u^n(s)+f(s,w^n_s)\right]\mathrm{d}s+\int_{0}^{T}\mathscr{R}(T-s)\left[f'_1(s)+f_2(s)\right]\mathrm{d}s\nonumber\\&=\mathscr{R}(T)\psi(0)+\int_{0}^{T}\mathscr{R}(T-s)f(s,w^n_s)\mathrm{d}s+\int_{0}^{T}\mathscr{R}(T-s)\left[f'_1(s)+f_2(s)\right]\mathrm{d}s\nonumber\\&\quad+\int_{\tau_n}^{T}\mathscr{R}(T-s)\mathrm{B}v^n(s-\tau_n)\mathrm{d}s.
		\end{align*}
		Finally, we evaluate
		\begin{align*}
			&\left\|w_n(T)-\zeta_1\right\|_{\mathbb{W}}\nonumber\\&=\big\|\mathscr{R}(T)\psi(0)+\int_{0}^{\tau_n}\mathscr{R}(T-s)\left[f(s,w^n_s)+f'_1(s)+f_2(s)\right]\mathrm{d}s\nonumber\\&\quad+\int_{\tau_n}^{T}\mathscr{R}(T-s)\left[f(s,w^n_s)+f'_1(s)+f_2(s)\right]\mathrm{d}s+\int_{\tau_n}^{T}\mathscr{R}(T-s)\mathrm{B}v^n(s-\tau_n)\mathrm{d}s-\zeta_1\big\|_{\mathbb{W}}\nonumber\\&=\big\|\left(\mathscr{R}(T)-\mathscr{R}(T-\tau_n)\mathscr{R}(\tau_n)\right)\psi(0)+\mathscr{R}(T-\tau_n)\mathscr{R}(\tau_n)\psi(0)+\int_{0}^{\tau_n}\mathscr{R}(T-s)f(s,z_s)\mathrm{d}s\nonumber\\&\quad+\int_{0}^{\tau_n}\mathscr{R}(T-s)\left[f'_1(s)+f_2(s)\right]\mathrm{d}s+\int_{\tau_n}^{T}\mathscr{R}(T-s)\mathrm{B}v^n(s-\tau_n)\mathrm{d}s\nonumber\\&\quad+\mathscr{R}(T-\tau_n)\int_{0}^{\tau_n}\mathscr{R}(\tau_n-s)\left[f(s,z_s)+f'_1(s)+f_2(s)\right]\mathrm{d}s\nonumber\\&\quad-\mathscr{R}(T-\tau_n)\int_{0}^{\tau_n}\mathscr{R}(\tau_n-s)\left[f(s,z_s)+f'_1(s)+f_2(s)\right]\mathrm{d}s\nonumber\\&\quad+\int_{\tau_n}^{T}\mathscr{R}(T-s)\left[f(s,w^n_s)+f'_1(s)+f_2(s)\right]\mathrm{d}s-\zeta_1\big\|_{\mathbb{W}}\nonumber\\&\le\left\|\left(\mathscr{R}(T)-\mathscr{R}(T-\tau_n)\mathscr{R}(\tau_n)\right)\psi(0)\right\|_{\mathbb{W}}\!\!+\!\left\|\mathscr{R}(T-\tau_n)z^n+\!\!\int_{0}^{T-\tau_n}\!\!\!\!\!\!\!\!\!\mathscr{R}(T-\tau_n-s)\mathrm{B}v^n(s)\mathrm{d}s-\zeta_1\right\|_{\mathbb{W}}\nonumber\\&\quad+\int_{0}^{\tau_n}\left\|\left[\mathscr{R}(T-s)-\mathscr{R}(T-\tau_n)\mathscr{R}(\tau_n-s)\right]\left[f(s,z_s)+f'_1(s)+f_2(s)\right]\right\|_{\mathbb{W}}\mathrm{d}s\nonumber\\&\quad+\int_{\tau_n}^{T}\left\|\mathscr{R}(T-s)\left[f(s,w^n_s)+f'_1(s)+f_2(s)\right]\right\|_{\mathbb{W}}\mathrm{d}s\to 0 \ \mbox{as}\ n\to\infty,
		\end{align*}
		where we have applied the convergence \eqref{eq5}, Lemma \ref{lem2.21} and the Lebesgue dominated convergence theorem. Therefore, the system \eqref{SEq} is approximately controllable.
	\end{proof}
	\section{Application}\label{sec5}\setcounter{equation}{0}
	In this section, we apply our findings to examine the approximate controllability of a neutral integro-differential equation that arises in the theory of heat conduction of materials with fading memory. The functional settings in the given example consider in the state space $\mathrm{L}^p([0,\pi];\mathbb{R})$ for $p\in[2,\infty)$, and the control space $\mathrm{L}^2([0,\pi];\mathbb{R})$ as discussed in \cite{MTM-20}.
	\begin{Ex}\label{ex1} Consider the following control system:
		\begin{equation}\label{57}
			\left\{
			\begin{aligned}
				&\frac{\partial}{\partial t}\left[w(t,\xi)+\int_{-\infty}^{t}(t-s)^{\alpha}e^{-\kappa(t-s)}w(s,\xi)\mathrm{d}s\right]\\&=\frac{\partial^2w(t,\xi)}{\partial \xi^2}\!+\!\eta(t,\xi)\!+\!\int_{-\infty}^{t}\!\!\!\!\!e^{-\mu(t-s)}w(s,\xi)\mathrm{d}s\!+\!\int_{-\infty}^{t}\!\!\!\!\!h(t-s)w(s,\xi)\mathrm{d}s, \ t\in (0,T], \xi\in[0,\pi], \\
				&w(t,0)=w(t,\pi)=0, \  t\in J=[0,T],\\
				&w(\theta,\xi)=\psi(\theta,\xi), \ \theta\in(-\infty, 0],\xi\in[0,\pi],
			\end{aligned}
			\right.
		\end{equation}
		where $\alpha\in(0,1), \kappa, \mu$ are positive constants and $h:[0,\infty)\to\mathbb{R}$ and $\psi:(-\infty,0]\times[0,\pi]\to\mathbb{R}$ are appropriate functions. The function $\eta: J\times[0,\pi]\to\mathbb{R}$ is square integrable in $t$ and $\xi$. 
	\end{Ex}
	To transform the above system in the abstract form \eqref{SEq}, we take a state space as a reflexive Banach space $\mathbb{W}_{p}=\mathrm{L}^p([0,\pi];\mathbb{R})$ with $p\in[2,\infty)$ and the control space $\mathbb{U}=\mathrm{L}^2([0,\pi];\mathbb{R})$. Choose the phase space $\mathfrak{B}=C_{0}\times\mathrm{L}^1_\omega(\mathbb{W})$. Note that the dual space of $\mathbb{W}_p$ is $\mathbb{W}_p^*=\mathrm{L}^{\frac{p}{p-1}}([0,\pi];\mathbb{R})$ which is uniform convex. We define the operator
	$\mathrm{A}_p:D(\mathrm{A}_p)\subset \mathbb{W}_p\to\mathbb{W}_p$ as
	$$\mathrm{A}_pg=g'',\ \ D(\mathrm{A}_p)=\mathrm{W}^{2,p}([0,\pi];\mathbb{R})\cap\mathrm{W}^{1,p}_0([0,\pi];\mathbb{R}).$$
	The operator $\mathrm{N}_p(t)g=e^{-\mu t}g$ for $g\in\mathbb{W}_p$ and $\mathrm{G}_p(t)g=t^{\alpha}e^{-\kappa t}g$ for $g\in\mathbb{W}_p$. Moreover, the operator $\mathrm{A}_p$ can be written as 
	\begin{align*}
		\mathrm{A}_pg=\sum_{k=1}^{\infty}-k^2\langle g, \nu_k \rangle \nu_k,\ g\in D(\mathrm{A}_p),
	\end{align*}
	where $ \nu_k(\xi)=\sqrt{\frac{2}{\pi}}\sin(k\xi)$ and $\langle g, v_k \rangle:=\int_{0}^{\pi}g(\xi)v_k(\xi)\mathrm{d}\xi$.
	\vskip 0.1 cm
	\noindent\textbf{Step 1:} \emph{Resolvent operator of linear problem.} The operator $\mathrm{A}_p$ with domain $D(\mathrm{A}_p)$ for any $p\in [2,\infty)$ is the infinitesimal generator of a strongly continuous semigroup on $\mathbb{W}_p$ and satisfying the estimate
	\begin{align*}
		\left\|\mathrm{R}(\lambda,\mathrm{A}_p)\right\|_{\mathcal{L}(\mathrm{L}^p)}\le \frac{1}{\lambda}, \ \mbox{for all}\ \lambda\in \textbf{C}\setminus \{-n^2 : n\in\mathbb{N}\},
	\end{align*}
	see, application section of \cite{IMA2022}. Hence, the semigroup is analytic followed by Theorem 5.2, \cite{SPN1983}. Thus, the condition \textbf{\textit{(Cd1)}} holds.
	
	Let us now verify the condition \textbf{\textit{(Cd2)}}. For this, we first compute
	\begin{align*}
		\left\|\mathrm{G}_p(t)g-\mathrm{G}_p(s)g\right\|_\mathbb{W}&=\left\|t^{\alpha}e^{-\kappa t}g-s^{\alpha}e^{-\kappa s}g\right\|_{\mathbb{W}}\nonumber\\&\le\left\|t^{\alpha}e^{-\kappa t}g-t^{\alpha}e^{-\kappa s}g\right\|_{\mathbb{W}}+\left\|t^{\alpha}e^{-\kappa s}g-s^{\alpha}e^{-\kappa s}g\right\|_{\mathbb{W}}\nonumber\\&=t^{\alpha}\left\|e^{-\kappa t}g-e^{-\kappa s}g\right\|_{\mathbb{W}}+|t^{\alpha}-s^{\alpha}|\left\|e^{-\kappa s}g\right\|_{\mathbb{W}}\nonumber\\&\to 0 \ \mbox{as} \ t\to s,
	\end{align*}
	for all $t,s\in [0,\infty)$ and $g\in\mathbb{W}$. Consequently, the operator $\mathrm{G}_p(\cdot)$ is strongly continuous. Next, the Laplace transformation of the operator $\mathrm{G}_p(\cdot)$ is given as
	\begin{align*}
		\widehat{\mathrm{G}_p}(\lambda)=\frac{\Gamma(\alpha+1)}{(\lambda+\kappa)^{\alpha+1}}\mathrm{I},\ \mathrm{Re}(\lambda)>0.
	\end{align*}
	In view of above expression, one can easily say that the family $\{\widehat{\mathrm{G}_p}(\lambda)g:\mathrm{Re}(\lambda)>0 \}$ is absolutely convergent for $g\in\mathbb{W}$. Moreover, we can also extend $\widehat{\mathrm{G}_p}(\lambda)$ for $\lambda\in\Lambda_{\nu} \ \mbox{with}\ \nu\in(\pi/2,\pi)$. Hence, the above facts ensures that the condition \textbf{\textit{(Cd2)}} is satisfied with $N_1=\Gamma(\alpha+1)$ and $N_2=\frac{\Gamma(\alpha+1)}{\kappa^\alpha}$.	The Laplace transformation $\widehat{\mathrm{N}_p}(\lambda)$ of the operator $\mathrm{N}_p(\cdot)$ is given by
	\begin{align*}
		\widehat{\mathrm{N}_p}(\lambda)=\frac{1}{(\lambda+\mu)}\mathrm{I},\ \lambda\in\Lambda_{\frac{\pi}{2}}.
	\end{align*}
	It clear form the above expression, the operator $\widehat{\mathrm{N}_p}(\lambda)$ can be extend for $\lambda\in\Lambda_{\nu} \ \mbox{with}\ \nu\in(\pi/2,\pi)$. Thus, the condition \textbf{\textit{(Cd3)}} is fulfilled with $\varPi(t)=e^{-\mu t}$.
	Finally, in order to verify the condition \textbf{\textit{(Cd4)}}, we choose $Y=C_0^{\infty}([0,\pi];\mathbb{R})$ (space of infinitely differentiable functions with compact support $[0,\pi]$). One can easily identify that the space $Y$ is dense in $\mathrm{W}^{2,p}([0,\pi];\mathbb{R})\cap\mathrm{W}^{1,p}_0([0,\pi];\mathbb{R})$ under the graph norm. Hence, the condition \textbf{\textit{(Cd4)}} is satisfied.
	Therefore, the abstract form of linear system corresponding to \eqref{57} has a resolvent operator $\mathcal{R}_p(\cdot)$ on $\mathbb{W}_p$.
	\vskip 0.1 cm
	\noindent\textbf{Step 2:} \emph{Functional setting and approximate controllability.} 
	Let us take $w(t)(\xi):=w(t,\xi)$ for $t\in J $ and $ \xi\in[0,\pi]$
	and the function $\psi:(-\infty,0]\rightarrow\mathbb{W}_p$  as
$ \psi(t)(\xi)=\psi(t,\xi),\ \xi\in[0,\pi]$
	We now assume the following:
	\begin{itemize}
		\item [(a)] The function $h(\cdot)$ is continuous and $K_f=\sup\limits_{s\in(-\infty,0] }\frac{|h(-s)|}{w(s)}<\infty$.
		\item [(b)] The function $\psi, \mathrm{A}\psi\in\mathfrak{B}$ and the values $\sup\limits_{s\in(-\infty,0)}\frac{e^{\mu s}}{w(s)}$ and  $\sup\limits_{t\in[0,T]}\left[\sup\limits_{s\in(-\infty,0)}\frac{(t-s)^{\alpha}e^{\kappa s}}{w(s)}\right]$ are finite.
		\item [(c)] The expression $\sup\limits_{t\in[0,T]}\left[\sup\limits_{s\in(-\infty,0)}\frac{e^{\kappa s}}{(t-s)^{1-\alpha}w(s)}\right]<\infty$.
	\end{itemize}
	Under the conditions (a), (b) and (c) the functions $f:[0,T]\times\mathfrak{B}\to\mathbb{W}_p$ and $f_1,f_2:[0,T]\to\mathbb{W}_p$ given by
	\begin{align}
		f(t,\phi)(\xi)&=\int_{-\infty}^{0}h(-s)\phi(s,\xi)\mathrm{d}s,
		\qquad f_1(t)\xi=\int_{-\infty}^{0}(t-s)^{\alpha}e^{-\kappa(t-s)}\psi(s,\xi)\mathrm{d}s,
		\nonumber\\f_2(t)\xi&=\int_{-\infty}^{0}e^{-\mu(t-s)}\mathrm{A}\psi(s,\xi)\mathrm{d}s,\nonumber
	\end{align}
	are well defined. Moreover, from condition (a), one can easily see that $f(\cdot)$ is continuous and $\left\|f\right\|_{\mathcal{L}(\mathfrak{B};\mathbb{W}_p)}\le K_f$. Thus, the Assumption \eqref{ass2.4} \textit{(H2)} are satisfied. Furthermore, the condition (b) and (c) guarantees that the functions $f_1\in C^{1}([0,T];\mathbb{W}_p)$ and $ f_2\in C([0,T];\mathbb{W}_p)$. 
	
	Next, the resolvent operator $\mathrm{R}(\lambda,\mathrm{A}_p)$ is compact for some $\lambda\in\rho(\mathrm{A}_p)$ (see, application section of \cite{IMA2022}). Consequently, Assumption  \eqref{ass2.4} \textit{(H1)} is followed.
		The operator $\mathrm{B}:\mathbb{U}\to\mathbb{W}_p$ is defined as  $$\mathrm{B}u(t)(\xi):=\eta(t,\xi)=\int_{0}^{\pi}H(\zeta,\xi)u(t)(\zeta)\mathrm{d}\zeta, \ t\in J,\ \xi\in [0,\pi],$$ with kernel $H\in\mathrm{C}([0,\pi]\times[0,\pi];\mathbb{R})$ and $H(\zeta,\xi)=H(\xi,\zeta),$ for all $\zeta,\xi\in [0,\pi]$.  Thus, the operator $\mathrm{B}$ is bounded (see, application section of \cite{IMA2022}). We can choose the kernel $H(\cdot,\cdot)$ such that the operator $\mathrm{B}$ is injective. In particular, $H(\xi,\zeta)=\min\{\xi,\zeta\}, \xi,\zeta\in [0,\pi]$.
	
	Using the above expressions, we can transform the system \eqref{57} into an abstract form as presented in \eqref{SEq} which satisfy all the assumptions. It is remaining to verify that the corresponding linear system of \eqref{SEq} is approximately controllable. To accomplish this, we take $\mathrm{B}^*\mathcal{R}_p(T-t)^*w^{*}=0,$ for any $w^{*}\in\mathbb{W^{*}_p}$. Then we have
	\begin{align*}
		\mathrm{B}^*\mathcal{R}_p(T-t)^*w^{*}=0\Rightarrow \mathcal{R}_p(T-t)^*w^{*}=0\Rightarrow w^{*}=0,
	\end{align*}
	and hence the linear system corresponding to \eqref{SEq} is approximately controllable is followed by Remark \ref{rem3.4}. Finally, by invoking Theorem \ref{thm4.4}, we conclude that the semilinear system \eqref{SEq} (equivalent to system \eqref{57}) is approximately controllable.
	\section{Concluding Remarks} In this study, we initiated our discussion to prove some important properties associated with the resolvent family $\mathcal{R}(\cdot)$ as defined in \eqref{RSL}.  Subsequently, we discussed the approximate controllability problem for the linear system \eqref{LEq1}. This investigation involve the optimization problem \eqref{opt} and finding the expression of the optimal control (see, lemma \ref{lem3.2}). Furthermore, we developed the existence of a mild solution of the neutral intego-differential equation \eqref{SEq} employing Schauder's fixed point theorem. We also formulated sufficient conditions for the approximate controllability of \eqref{SEq} within a reflexive Banach space having uniform convex dual. Additionally, we demonstrated the approximate controllability of the system \eqref{SEq} in a general Banach space, assuming a Lipschitz type condition on the nonlinear term. Finally, we applied our findings to determine the approximate controllability of the neutral integro-differential equation relevant to the theory of heat conduction of material with fading memory. In future aspects, we aim to explore this study in the framework of fractional order integro-differential equations and inclusions.
	
	\medskip\noindent	{\bf Acknowledgments:} The authors would like to thank the referees for their valuable suggestions and comments on the first version of this  manuscript which helped us to improve the presentation of the article. The authors would like to thank the Department of Science $\&$ Technology, Government of India for the support under FIST (No:SR/FST/MS-II/2021/97) to the Department of Mathematics, Indian Institute of Science (IISC), Bangalore, India. The first  author would also like to thank Indian Institute of Science for providing him the Institute of Eminence (IoE) fellowship.

\end{document}